\documentclass[11pt]{article}
\usepackage{amsthm, amssymb, bm}
\usepackage{soul, color}
\usepackage[]{amsmath}
\usepackage[]{amsfonts}
\usepackage[]{fancyhdr}
\usepackage[]{graphicx,subfigure}
\usepackage{algorithm}
\usepackage{algorithmic}
\graphicspath{{figures/}}

\newtheorem{theorem}{Theorem}[section]
\newtheorem{lemma}[theorem]{Lemma}
\newtheorem{proposition}[theorem]{Proposition}
\newtheorem{corollary}[theorem]{Corollary}

\def \Cm {\mathbb{C}}

\def \Rm {\mathbb{R}}
\def \Sm {\mathbb{S}}

\def \Zm {\mathbb{Z}}

\def\C{\mathcal{C}}
\def\D{\mathcal{D}}

\def\K{\mathcal{K}}
\def\L{\mathcal{L}}

\def\diam{ {{\rm diam}}}

\newcommand{\tf}{\tilde{f}}

\newcommand{\wtW}{ {\widetilde W} }

\newcommand{\where}{\quad\text{ where }}
\newcommand{\qandq}{\quad\text{ and }\quad}

\newcommand{\cout}[1]{}

\newcommand{\x}{{\mathrm x}}
\newcommand{\y}{{\mathrm y}}

\newcommand{\tu}{{\tilde{u}}}

\newcommand{\dprod}[2]{\langle #1, #2 \rangle}
\newcommand{\sgn}[1]{\text{ sgn}(#1)}

\title{On reconstruction formulas for the ray transform acting on symmetric differentials on surfaces}
\author{Fran\c cois Monard\thanks{Department of Mathematics, University of Washington. email: fmonard@uw.edu}}

\begin{document}
\maketitle

\begin{abstract}
    The present article proposes a partial answer to the explicit inversion of the tensor tomography problem in two dimensions, by proving injectivity over certain kinds of tensors and providing reconstruction formulas for them. These tensors are symmetric differentials of any order as well as other types obtained after taking their transverse covariant derivative once. Such reconstruction formulas require introducing additional types of fiberwise Hilbert transforms which satisfy a convenient generalization of the Pestov-Uhlmann commutator formula \cite{Pestov2005}. 

    Numerical simulations in {\tt MatLab} are provided using the author's code presented in \cite{Monard2013}, justifying the exactness of the formulas in some cases of simple and near-simple metrics, and displaying issues encountered as one increases either curvature, lack of simplicity, or tensor order, in all of which cases the formulas derived in the first part become theoretically insufficient. 
\end{abstract}


\section{Introduction}

Integral transforms arise in many areas of imaging sciences, including medical and seismic imaging. The most famous example is the two-dimensional Radon transform of a function, with applications in X-ray computerized tomography:
\begin{align*}
    Rf(s,\theta) = \int_{L(s,\theta)} f\ d\ell, \quad (s,\theta)\in \Rm\times \Sm^1,
\end{align*}
where $L(s,\theta) := \{ (x_1,x_2) \in \Rm^2: x_1\cos\theta + x_2\sin\theta = s \}$. Such integral transforms can be generalized to $(M,g)$ a non-trapping Riemannian surface with boundary, where the curves of integrations are the geodesics for the metric $g$, and to more general integrands expressed here as smooth functions defined on the unit circle bundle of $M$ (denoted $SM$), giving rise to the following geodesic X-ray transform
\begin{align}
    If(x,v) = \int_0^{\tau(x,v)} f(\varphi_t(x,v))\ dt, \quad (x,v)\in \partial_+ SM,
    \label{eq:geoRT}
\end{align}
where $\varphi_t(x,v) = (\gamma_{x,v}(t), \dot\gamma_{x,v}(t))$ denotes the geodesic flow, $\tau(x,v)$ is the first time of exit of the geodesic $\gamma_{x,v}$ and $\partial_+ SM:= \{(x,v): x\in \partial M, |v|=1, v\cdot\nu_x >0 \}$ is the influx boundary of $SM$ ($\nu_x$ is the unit-inner normal to $\partial M$ at $x$). Such a setting covers the case of the X-ray transforms of symmetric covariant $k$-tensors for any integer $k\ge 0$, in which case the integrand in \eqref{eq:geoRT} takes the form
\begin{align*}
    f(\varphi_t(x,v)) = f(\gamma_{x,v}(t), \dot\gamma_{x,v}(t)^{\otimes k}).
\end{align*}
Given $(M,g)$ a non-trapping surface with boundary and its ray transform defined on $\partial_+ SM$, questions of theoretical interest are injectivity (or characterization of the lack thereof), stability, range characterization and reconstruction algorithms. 

The case of functions ($k=0$) has applications in Computerized Tomography and injectivity was first solved in \cite{Mukhometov1975} for simple metrics. Integral transforms of vector fields ($k=1$) arise in Doppler ultrasound tomography, and injectivity over solenoidal vector fields was first established in \cite{Anikonov1997} in the simple case. Pestov and Uhlmann then derived in \cite{Pestov2004} Fredholm equations leading to a reconstruction procedure up to smooth error for both cases $k=0,1$. These formulas become exact when curvature is constant, and can be made exact when curvature is close enough to constant \cite{Krishnan2010}. The case $k=2$ arises in the deformation (or linearized) boundary rigidity problem, which involves the inversion of the ray transform of a second-order tensor (a perturbation of a known background metric tensor). This case is treated in \cite{Sharafudtinov2007} via variations of Dirichlet-to-Neumann maps. The case of fourth-order tensors ($k=4$) finds applications in travel time tomography in slightly anisotropic elastic media \cite[Ch. 7]{Sharafudtinov1994}, where the complex amplitude of compressional waves gives, to first approximation, the ray transform of a symmetric $4$-tensor. Microlocal techniques in \cite{Stefanov2004} show in general that the normal operator $N = I^\star I$ is a smoothing pseudo-differential operator of order $-1$, and the authors give there a microlocal inversion over tensors of arbitrary order for simple metrics, thereby giving a microlocal procedure to resolve singularities. While the work previously described establishes a stability result, the question of solenoidal injectivity for tensors of any order on simple surfaces was recently established in \cite{Paternain2011a}, and a range characterization was given by the same authors in \cite{Paternain2013a}, proving in passing more general results for attenuated ray transforms. The interested reader is also invited to read the review \cite{Paternain2013} for more general settings and open problems.


\medskip
In the present article, we consider inversion formulas for the X-ray transform over elements of the form (i) $f$ and (ii) $X_\perp h$ ($X_\perp$ denotes the transverse geodesic flow, defined in Sec. \ref{sec:SM}), where $f,h$ are smooth functions on $SM$ belonging to the space
\begin{align*}
    \Omega_k := \ker (\partial_\theta - ik Id) \cap \C^\infty(SM), 
\end{align*}
for some fixed $k\in \Zm$ ($\partial_\theta$ denotes the so-called {\em vertical derivative}, defined in Sec. \ref{sec:SM}), and $h$ vanishes at the boundary $\partial M$. The case $k=0$ corresponds to functions and solenoidal vector fields treated in \cite{Pestov2004}. The integrands of the kind (i) mentioned above are particular kinds of symmetric tensors in a basis of holomorphic coordinates $z = x + i y$ for the underlying Riemann surface (see \cite{Farkas1992}): an element of $\Omega_k$ can be obtained by considering tensors of the form $g(x,y)\ \sigma (dz^{\otimes p} \otimes d\bar{z}^{\otimes q})$ with $p-q = k$ for some function $g$ defined on $M$ and where $\sigma$ denotes symmetrization. When $p$ or $q$ equals zero, such tensors are symmetric differentials. On the other hand, elements of the kind (ii) arise naturally by duality of the first problem, in a similar way to \cite{Pestov2004} when $k=0$. For both integrands considered, the recent $s$-injectivity result in \cite{Paternain2011a} allows us to first prove that the problem considered here is, in fact, (completely) injective regardless of whether these integrands are solenoidal, which not all of them are. This injectivity is to be expected as, unlike full symmetric $m$-tensor fields, these elements are only spanned by two dimensions of data, which makes the problem considered formally determined. 


Reconstruction formulas for solenoidal tensors of order $k>2$ in the non-Euclidean setting are still under active study (in the Euclidean case, Sharafutdinov settles the question in any dimension $n\ge 2$ in \cite[Theorem 2.12.2]{Sharafudtinov1994}). In that regard and since any smooth symmetric tensor decomposes orthogonally into $\bigoplus_{l\in \Zm} \Omega_l$, inversion of the ray transform over $\Omega_k$ can be seen as a building block toward such reconstruction formulas. The range characterization of the ray transform over $\Omega_k$ was recently used in \cite{Paternain2013a} as a tool for characterizing the range of the ray transform over symmetric $k$-tensors. 


Inverting $If$ for $f\in \Omega_k$ can also be recast as inverting the attenuated ray transform of a function for a particular unitary connection, a purely imaginary ``attenuation'' term. In the case of a real-valued attenuation, reconstruction algorithms for such transforms were provided in \cite{Salo2011}. A novelty here is that the inversion in this case does not use holomorphic integrating factors. The main idea here consists in conjugating the Hilbert transform (by appropriate powers of a non-vanishing abelian differential) rather than conjugating the geodesic flow, and using some ideas of \cite{Pestov2004} to derive Fredholm equations for the unknowns. These equations in turn allow the reconstruction of $k$-differentials and their transverse derivative up to a smooth, explicitely given error in the case of {\em simple} metrics. In addition, we prove that said error operators become contractions when the curvature and its gradient are small enough in uniform norm, so that the Fredholm equations become exact reconstruction formulas via Neumann series, generalizing a prior result by Krishnan \cite{Krishnan2010} in the case $k=0$. Such conditions on curvature become more stringent as the tensor order $k$ increases. 
Finally, we present a numerical implementation of the reconstruction algorithms derived, using a {\tt MatLab} code previously developed by the author and presented in \cite{Monard2013} in the context of inverting the ray transform over functions and solenoidal vector fields. We illustrate how the reconstruction algorithms proposed behave depending on curvature, the order of the differential, as well as the simplicity of the metric. 

\paragraph{Outline.} The rest of the paper is structured as follows. Section \ref{sec:SM} recalls important facts about the unit circle bundle, after which it is proved in Section \ref{sec:injectivity} that both problems considered are injective. In Section \ref{sec:Hk} we introduce generalizations of the fiberwise Hilbert transform, obtained by conjugating it with powers of a non-vanishing abelian differential, and we generalize the Pestov-Uhlmann commutator formula to these operators (Lemma \ref{lem:commutatork}). Section \ref{sec:Fredholm} provides Fredholm equations (Thm \ref{thm:main}) which in turn lead to exact reconstruction procedures under smallness assumptions on the curvature (Prop. \ref{prop:Wkest}, Cor. \ref{cor:contraction}). Section \ref{sec:numerics} covers numerical examples. 

\section{The geometry of the unit circle bundle} \label{sec:SM}

Let us denote the geodesic flow $X = \frac{\partial}{\partial t}\varphi_t|_{t=0}$ on $SM$. Assuming to work on a simply connected surface, there exists a global chart of isothermal coordinates (see \cite{Ahlfors1966}) $\x = (x,y)$ with frame $(\partial_x,\partial_y)$. At each point, the tangent unit circle is parameterized by an angle $\theta$ denoting the angle of the tangent vector w.r.t. to a fixed vector, e.g. $\partial_x$. In these coordinates the metric is isotropic and reads $ds^2 = g(x,y) (dx^2  + dy^2) = e^{2\lambda(x,y)} (dx^2  + dy^2)$. We will denote by $\gamma_{\x,\theta}(t)$ the unique unit-speed geodesic with initial conditions 
\begin{align*}
    \gamma_{\x,\theta}(0) = \x \qandq \dot\gamma_{\x,\theta}(0) = e^{-\lambda(\x)} \binom{\cos\theta}{\sin\theta}, \qquad (\x,\theta) \in M\times\Sm^1.     
\end{align*}
Defining $\tau(\x,\theta)$ as the first exit time of the geodesic $\gamma_{\x,\theta}$, we obtain a geodesic mapping defined on the set  
\begin{align}
    \D := \{ (\x,\theta,t) : (\x,\theta)\in SM,\ -\tau(\x,\theta+\pi) < t < \tau(\x,\theta) \}.
    \label{eq:D}
\end{align}

There exists a circle action on the unit tangent bundle $SM$, whose infinitesimal generator, also called {\em vertical vector field}, is given by $V \equiv\frac{\partial}{\partial\theta}$. From $X,V$ one may construct a global frame of $T(SM)$ by constructing the vector field $X_\perp := [X,V]$, where $[\cdot,\cdot]$ stands for the Lie bracket, or commutator, of two vector fields. One also has the additional structure equations $[V,X_\perp] = X$ and $[X,X_\perp] = \kappa V$, with $\kappa$ the Gaussian curvature. In isothermal coordinates $(x, y, \theta)$, these vector fields read
\begin{align}
    \begin{split}
	X &= e^{-\lambda} \left( \cos\theta \partial_x + \sin\theta \partial_y + (-\sin\theta \partial_x\lambda + \cos\theta\partial_y\lambda)\ \partial_\theta \right), \\
        X_\perp &= - e^{-\lambda} \left( -\sin\theta\partial_x + \cos\theta\partial_y - (\cos\theta\partial_x\lambda + \sin\theta\partial_y\lambda)\ \partial_\theta \right).	
    \end{split}
    \label{eq:XXperpiso}  
\end{align}
We can then define a Riemannian metric on $SM$ by declaring $(X,X_\perp,V)$ to be an orthonormal basis and the volume form of this metric will be denoted by $d\Sigma^3$ (in isothermal coordinates, this form becomes $e^{2\lambda}\ dx\ dy\ d\theta$). The fact that $(X,X_\perp,V)$ are orthonormal together with the structure equations implies that the Lie derivative of $d\Sigma^3$ along the three vector fields vanishes, therefore these vector fields are volume preserving. Introducing the inner product 
\begin{align*}
    (u,v) = \int_{SM} u\bar{v}\ d\Sigma^3, \qquad u,v:SM\to \Cm,    
\end{align*}
with the bar denoting conjugation, the space $L^2(SM,\Cm)$ decomposes orthogonally as a direct sum
\begin{align}
    L^2(SM, \Cm) = \bigoplus_{k\in \Zm} H_k,
    \label{eq:L2decomp}
\end{align}
where $H_k$ is the eigenspace of $-iV$ corresponding to the eigenvalue $k$. As in the introduction, we also denote $\Omega_k := \C^\infty(SM) \cap H_k$. A smooth function $u:SM\to \Cm$ has a Fourier series expansion
\begin{align*}
    u = \sum_{k=-\infty}^{\infty} u_k(\x,\theta), \where \qquad u_k(\x,\theta) = e^{ik\theta} \tu_k (\x), \qquad \tu_k(\x) = \frac{1}{2\pi} \int_{\Sm^1} u(\x,\theta) e^{-ik\theta}\ d\theta.
\end{align*}
Such functions admit an even/odd decomposition w.r.t. to the involution $\theta\mapsto \theta+\pi$, denoted
\begin{align}
    u = u_+ + u_-, \where\quad  u_+ := \sum_{k \text{ even}} u_k \qandq u_- := \sum_{k \text{ odd}} u_k.
    \label{eq:oddeven}
\end{align}
An important decomposition of $X$ and $X_\perp$ due to Guillemin and Kazhdan (see \cite{Guillemin1980}) is given by defining $\eta_{\pm}:= \frac{X\pm iX_\perp}{2}$, so that one has the following decomposition 
\[ X = \eta_+ + \eta_- \qandq X_\perp = \frac{1}{i} (\eta_+ - \eta_-), \]
with the important property that $\eta_{\pm}:\Omega_k\to \Omega_{k\pm 1}$ for any $k\in \Zm$. Moreover, it should be noted that the operators $\eta_{\pm}$ are {\em elliptic} and that, since $M$ is simply connected, both problems
\begin{align*}
    \left(\eta_+ u = 0\quad (SM), \quad u|_{\partial M} = 0\right) \qandq \left(\eta_- u = 0\quad (SM), \quad u|_{\partial M} = 0\right),
\end{align*}
only admit the trivial solution $u\equiv 0$. Using the decomposition above, it is clear that both $X$ and $X_\perp$ map odd functions on $SM$ into even ones and vice-versa. 

\section{Injectivity of the ray transform over $\Omega_k$ and $X_\perp \Omega_k$} \label{sec:injectivity}

Recall the recent result by Paternain-Salo-Uhlmann on s-injectivity of the ray transform over tensors of any order:
\begin{theorem}\label{thm:PSU}\cite[Theorem 1.1]{Paternain2011a} Let $(M,g)$ be a simple 2D manifold and let $m\ge 0$. If $f$ is a smooth symmetric $m$-tensor field on $M$ which satisfies $If = 0$, then $f = dh$ for some smooth symmetric $m-1$-tensor field $h$ on $M$ with $h|_{\partial M} = 0$. (If $m=0$ then $f=0$).        
\end{theorem}
Thm. \ref{thm:PSU} can be used to establish full injectivity of the ray transform over $\Omega_k$ and $X_\perp \Omega_k$. 

\begin{theorem}\label{thm:injectivity}
    Let $(M,g)$ be a simple 2D manifold and let $k$ any integer. 
    \begin{itemize}
	\item[(i)] If $f\in \Omega_k$ is such that $If = 0$, then $f = 0$. 
	\item[(ii)] If $f \in \Omega_k$ with $f|_{\partial M} = 0$ is such that $I[X_\perp f] = 0$, then $f = 0$.
    \end{itemize}    
\end{theorem}

\begin{proof}
    When $k=0$, (i) is due to \cite{Mukhometov1975} and (ii) is due to \cite{Anikonov1997}. We now treat the case $k>0$, the case $k<0$ being completely similar. \medskip

    \noindent {\bf Proof of (i).} $f \in \Omega_k$ is a particular case of symmetric $k$-tensor, so if $If = 0$, by Theorem \ref{thm:PSU} there exists $h$ a $k-1$-tensor such that $Xh = f$ on $SM$ and $h|_{\partial M} = 0$. That $h$ is a $k-1$ tensor means that it decomposes into 
    \begin{align*}
	h = h_{-k+1} + h_{-k+3} + \dots + h_{k-3} + h_{k-1}, \qquad h_l \in \Omega_l,
    \end{align*}
    and $h|_{\partial M} = 0$ implies that $h_l|_{ \partial M} = 0$ for every $l$. Writing $Xh = (\eta_+ + \eta_-) h = f$ and projecting this equation onto each $\Omega_l$, we obtain the relations:
    \begin{align*}
	\eta_- h_{-k+1} = 0, \quad \eta_- h_{-k+3} + \eta_+ h_{-k+1} = 0, \quad \dots, \quad \eta_- h_{k-1} + \eta_+ h_{k-3} = 0, \quad \eta_+ h_{k-1} = f.
    \end{align*}
    Since each $h_{l}$ vanishes at the boundary, the first equation entails that $h_{-k+1} = 0$, and the vanishing of all coefficients $h_l$ follows in cascade by using all equations above but the last one. Since $h_{k-1} = 0$, the last equation concludes that $f = \eta_+ h_{k-1} = 0$. \medskip 

    \noindent {\bf Proof of (ii).} For $f \in \Omega_k$, $X_\perp f$ is a particular case of symmetric $(k+1)$-tensor, so if $I[X_\perp f] = 0$, by Theorem \ref{thm:PSU} there exists $h$ a $k$-tensor such that $Xh = f$ on $SM$ and $h|_{\partial M} = 0$. That $h$ is a $k$-tensor means that it decomposes into 
    \begin{align*}
	h = h_{-k} + h_{-(k-2)} + \dots + h_{k-2} + h_{k}, \qquad h_l \in \Omega_l,
    \end{align*}
    and $h|_{\partial M} = 0$ implies that $h_l|_{ \partial M} = 0$ for every $l$. Projecting the equation $(\eta_+ + \eta_-) h = X_\perp f = \frac{1}{i} (\eta_+ - \eta_-) f$ onto each $\Omega_l$, we obtain the relations:
    \begin{align}
	\eta_- h_{-k} = 0, \quad \eta_- h_{-k+2} + \eta_+ h_{-k} = 0,\quad  \dots, \quad \eta_- h_{k-2} + \eta_+ h_{k-4} = 0,
	\label{eq:rel1}
    \end{align}
    and the last two equations read
    \begin{align}
	\eta_- h_{k} - \eta_+ h_{k-2} = - \frac{1}{i} \eta_- f \qandq \eta_+ h_k = \frac{1}{i} \eta_+ f. 
	\label{eq:rel2}	
    \end{align}
    Equations \eqref{eq:rel1} and the boundary conditions imply successively that $h_{-k} = h_{-k+2} = \dots = h_{k-2} = 0$. Therefore equations \eqref{eq:rel2} become
    \begin{align*}
	\eta_- (h_k + \frac{1}{i} f) = 0 \qandq \eta_+ (h_k - \frac{1}{i} f) = 0. 
    \end{align*}
    With the boundary conditions $(h_k \pm \frac{1}{i} f)|_{\partial M} = 0$, this implies $h_k \pm \frac{1}{i} f = 0$, which upon summing and substracting yields $h_k  = f = 0$. Theorem \ref{thm:injectivity} is proved.  
\end{proof}

\section{Conjugated Hilbert transforms and commutators} \label{sec:Hk}

In the $L^2(SM)$ decomposition \eqref{eq:L2decomp}, a diagonal operator of particular interest is the so-called fiberwise {\em Hilbert transform} $H:\C^\infty(SM)\to \C^\infty(SM)$, whose action is best described on each Fourier component by
\begin{align*}
    H u_k := -i \sgn {k} u_k, \quad k\in \Zm, \qquad \text{with the convention }\quad \sgn{0} = 0.
\end{align*}
Pestov and Uhlmann proved in \cite{Pestov2005} the commutator formula 
\[ [H,X] u = X_\perp u_0 + (X_\perp u)_0, \qquad u\in \C^\infty(SM), \]
which was a crucial tool in the derivation of reconstruction formulas for functions and solenoidal vector fields in \cite{Pestov2004}.

A generalization of this formula that is adapted to our problem, is obtained by shifting the Hilbert transform in frequency: define the shifted Hilbert transform $H_{(k)}$ on the Fourier components of $u$ as
\begin{align*}
  H_{(k)} u_l = -i \sgn{l-k} u_l, \quad l\in \Zm.
\end{align*}
This can be achieved by considering $h$ a non-vanishing section of $\Omega_1$ and defining 
\begin{align}
    H_{(k)} u := h^k H(h^{-k} u).
    \label{eq:Hk}
\end{align}
We then notice that this definition does not depend on the choice of $h$, as its action on the harmonic components of $u$ is described two equations above. In isothermal coordinates, one has the obvious choice $h(\x,\theta) = e^{i\theta}$. As in the case of $H\equiv H_{(0)}$, we now establish commutation relations.
\begin{lemma}\label{lem:commutatork}
    For every $u\in \C^\infty(SM)$, the following formulas hold
    \begin{align}
	[H_{(k)}, X] u &= X_\perp u_k + (X_\perp u)_k, \label{eq:commutatork} \\
	[H_{(k)}, X_\perp] u &= - X u_k - (X u)_k. \label{eq:commutatork2} 
    \end{align}        
\end{lemma}

\begin{proof}
    We will rely on the fact that 
    \begin{align*}
	H_{(k)}|_{\Omega_l} = \left\{ \begin{array}{cc}
	    0 & l=k, \\
	    -i\ \text{Id}|_{ \Omega_l} & l > k, \\
	    i\ \text{Id}|_{\Omega_l} & l < k,
	\end{array}
	\right.
    \end{align*}
    with the fact that $\eta_{\pm}:\Omega_p\to \Omega_{p\pm 1}$, we deduce that for any $u_n\in \Omega_n$, 
    \begin{align*}
	[H_{(k)},\eta_+] u_n = \left\{  \begin{array}{cc}
	    0 & n \notin \{ k-1, k \}, \\
	    H_{(k)}\eta_+ u_k = -i\ \eta_+ u_k & n=k, \\
	    H_{(k)}\eta_+ u_{k-1}  - \eta_+ H_{(k)} u_{k-1} = -i \eta_+ u_{k-1} & n=k-1,
	\end{array}
	\right.
    \end{align*}
    as well as 
    \begin{align*}
	[H_{(k)},\eta_-] u_n = \left\{ \begin{array}{cc}
	    0 & n \notin \{ k, k+1 \}, \\
	    H_{(k)}\eta_- u_k = i\ \eta_- u_k & n=k, \\
	    H_{(k)}\eta_- u_{k+1}  - \eta_- H_{(k)} u_{k+1} = i \eta_- u_{k+1} & n=k+1.
	\end{array}
	\right.
    \end{align*}
    Writing $u = \sum_{n=-\infty}^{\infty} u_n$, we now sum the relations above
    \begin{align*}
	[H_{(k)},X] \sum_{n\in\Zm} u_n &= \sum_{n\in \Zm} [H_{(k)},\eta_+] u_n + \sum_{n\in\Zm} [H_{(k)},\eta_-] u_n \\
	&= -i\ \eta_+ u_k -i \eta_+ u_{k-1} + i\ \eta_- u_k + i \eta_- u_{k+1} \\
	&= X_\perp u_k + (X_\perp u)_k. 
    \end{align*} 
    Similarly,
    \begin{align*}
	[H_{(k)},X_\perp] u &= i^{-1} \left( [H_{(k)},\eta_+]u - [H_{(k)},\eta_-]u  \right) \\
	&= - \eta_+ u_k - \eta_+ u_{k-1} - \eta_- u_k - \eta_- u_{k+1} \\
	&= - X u_k - (Xu)_k.
    \end{align*} 
    The proof is complete. 
\end{proof}

\section{Fredholm equations over $\Omega_k$ and $X_\perp \Omega_k$} \label{sec:Fredholm}

\subsection{Transport equations and past results (the case $k=0$)}

For $f\in L^2(SM)$, let us define $u^f(\x,\theta)$ to be the solution to the problem 
\begin{align}
    X u = -f, \qquad u|_{\partial_- SM} = 0,
    \label{eq:uf}
\end{align}
so that $u|_{\partial_+ SM} = If$. For $w$ defined on $\partial_+ SM$, we also denote by $w_\psi = w\circ\alpha\circ\psi$ the unique solution to the transport problem
\begin{align*}
    X u = 0, \qquad u|_{\partial_+ SM} = w,
\end{align*}
where $\alpha$ is the scattering relation and $\psi(\x,\theta) := (\gamma_{\x,\theta} (\tau(\x,\theta)), \dot\gamma_{\x,\theta} (\tau(\x,\theta)))\in \partial_- SM$ for any $(\x,\theta)\in SM$.

When $f\in \C_0^\infty (M)$, we define the operator $W f := (X_\perp u^f)_0$. It is shown in \cite{Pestov2004} that when the metric is simple, the operator $W$ can be extended as a smoothing operator $W:L^2(M)\to \C^\infty(M)$. Moreover, this operator vanishes identically if the scalar curvature is constant. The $L^2(M)$-adjoint operator $W^\star$ is given by 
\begin{align*}
  W^\star h := \left( u^{X_\perp h} \right)_0.
\end{align*}

Recall the following theorem due to Pestov and Uhlmann.
\begin{theorem}[Theorem 5.4 in \cite{Pestov2004}] \label{thm:PU}
    Consider $f\in L^2(M)$ and $h\in C_0^1(M)$ giving rise to the solenoidal vector field $X_\perp h$ and denote $I f$ and $I X_\perp h$ their respective X-ray transforms. Then one has the following two formulas
    \begin{align}
	f + W^2 f &= - (X_\perp w^{(f)}_\psi)_0, \quad\text{where}\quad  w^{(f)} := (H I f)_-, \label{eq:frc} \\
	h + (W^\star)^2 h &= - ( w^{(h)}_\psi )_0, \quad\text{where}\quad  w^{(h)} := (H I [X_\perp h])_+ . \label{eq:hrc} 
    \end{align}
\end{theorem}

Although the initial theorem is stated for a simple manifold, formula \eqref{eq:frc} only requires that the transport equation \eqref{eq:uf} be well-defined, and still holds true as long as $M$ is non-trapping. If the metric is simple, then \eqref{eq:frc}-\eqref{eq:hrc} both satisfy Fredholm alternatives. If the metric is not simple, then the operators $W,W^\star$ may no longer be smoothing operators. The proof of Theorem \ref{thm:PU} is covered as a particular case $k=0$ of Theorem \ref{thm:main} in the next section.

\subsection{Fredholm equations}

We now derive a generalization of the reconstruction formulas above, using the shifted Hilbert transforms introduced in Sec. \ref{sec:Hk}. Before stating the theorem, we introduce the generalization of the operator $W$ above. For any integer $k$, let us define the operator $W_k:\Omega_k\to \Omega_k$ by 
\begin{align}
    W_k f := (X_\perp u^f)_k, \qquad f \in \Omega_k,
    \label{eq:Wk}
\end{align}
whose properties will be studied in section \ref{sec:propWk}. We will show in particular that the operator $W_k^\star:\Omega_k\to \Omega_k$ defined by 
\begin{align}
    W_k^\star h := (u^{X_\perp h})_k, \qquad h\in \Omega_k, \quad h|_{\partial M} = 0,
    \label{eq:Wkstar}
\end{align}
is the adjoint operator of $W_k$ for the $L^2(SM)$ inner product structure on $H_k$. Moreover, when $(M,g)$ is simple, both operators will be shown to be compact.

\begin{theorem}\label{thm:main}
    Let $k\in \Zm$ and define $\sigma_k := +/-$ if $k$ is an even/odd integer, respectively. Consider $f \in \Omega_k$ and $h\in \Omega_k$ with $h|_{\partial M} =0$ giving rise to the tensor field $X_\perp h$ and denote $I f$ and $I [X_\perp h]$ their respective X-ray transforms. Then one has the following two formulas
    \begin{align}
	f + W_k^2 f &= - (X_\perp w^{(f)}_\psi)_k, \quad\text{where}\quad  w^{(f)} := H_{(k)} (I f)_{-\sigma_k}|_{\partial_+ SM}, \label{eq:fkrc} \\
	h + (W_k^\star)^2 h &= - ( w^{(h)}_\psi )_k, \quad\text{where}\quad  w^{(h)} := H_{(k)} I[X_\perp h]_{\sigma_k}|_{\partial_+ SM} . \label{eq:hkrc} 
    \end{align}
\end{theorem}

\begin{proof}
    {\bf Proof of \eqref{eq:fkrc}:} Let $f \in \Omega_k$ for some fixed $k$ and consider the transport equation 
    \begin{align}
	X u = - f, \qquad u|_{\partial_- SM} = 0,
	\label{eq:transk}
    \end{align}
    Upon splitting the above equation into even/odd parts $u = u_{\sigma_k} + u_{-\sigma_k}$, we have $X u^f_{-\sigma_k} = -f$ and $X u_{\sigma_k}^f = 0$. Applying $H_{(k)}$ to \eqref{eq:transk}, using the commutator \eqref{eq:commutatork}, splitting into odd and even parts and using the fact that $H_{(k)}f = 0$, we arrive at the equation
    \begin{align}
	\begin{split}
	    X (H_{(k)} u^f_{-\sigma_k}) &= - (X_\perp u^f)_k = -W_k f,  \\
	    H_{(k)} u_{-\sigma_k}^f|_{\partial_+ SM} &= w^f := H_{(k)} (I_k f)_{-\sigma_k}|_{\partial_+ SM} \quad \text{ (data)},
	\end{split}
	\label{eq:transHku}
    \end{align}
    where $W_k$ is defined in \eqref{eq:Wk}. Equation \eqref{eq:transHku} implies that $H_{(k)} u_{-\sigma_k}^f$ is nothing but 
    \begin{align*}
	H_{(k)} u_{-\sigma_k}^f = u^{W_k f} + w^f_\psi.
    \end{align*}
    Applying the operator $(X_\perp\cdot)_k$ to the relation above, we establish the relation
    \begin{align*}
	(X_\perp H_{(k)} u_{-\sigma_k}^f )_k = (X_\perp u^{W_k f})_k + (X_\perp w^f_\psi)_k = W_k^2 f + (X_\perp w^f_\psi)_k.
    \end{align*}
    Applying $H_{(k)}$ to \eqref{eq:transHku}, using the commutator relation again and the fact that $H_{(k)} W_k f = 0$ and $(H_{(k)}u_{-\sigma_k}^f)_k = 0$, we arrive at the relation
    \begin{align*}
	X( H_{(k)}^2 u_{-\sigma_k}^f) = - (X_\perp H_{(k)} u_{-\sigma_k}^f )_k.
    \end{align*}
    Combining the last two equations together and using the fact that 
    \begin{align*}
	X( H_{(k)}^2 u_{-\sigma_k}^f) = X ( - u_{-\sigma_k}^f + (u_{-\sigma_k}^f)_k ) = f, \qquad ( \text{using }(u_{-\sigma_k}^f)_k = 0)
    \end{align*}
    we arrive at \eqref{eq:fkrc}. \\

    {\bf Proof of \eqref{eq:hkrc}:} For a function $h\in \Omega_k$, let us consider the transport equation
    \begin{align*}
	Xu = -X_\perp h, \quad u|_{\partial_- SM} = 0.
    \end{align*}
    Applying $H_{(k)}$ to the above equation and using formulas \eqref{eq:commutatork}-\eqref{eq:commutatork2}, we arrive at
    \begin{align*}
	X(H_{(k)}u^{X_\perp h}) + X_\perp W_k^\star h + (X_\perp u^{X_\perp h})_k = -X_\perp  (H_{(k)}h ) + X (h_k) + (X h)_k,
    \end{align*}
    where $W_k^\star$ is defined in \eqref{eq:Wkstar}. Using the fact that $(Xh)_k = 0$ and $H_{(k)}h = 0$ and splitting into odd and even parts, we get the equation 
    \begin{align*}
	X (H_{(k)}u^{X_\perp h}_{\sigma_k} - h ) = - X_\perp (W_k^\star h),
    \end{align*}
    with boundary condition
    \begin{align*}
	H_{(k)}u^{X_\perp h}_{\sigma_k} - h|_{\partial_+ SM} = w^h := H_{(k)} (I(X_\perp h))_{\sigma_k} |_{\partial_+ SM}.
    \end{align*}
    We thus deduce that 
    \begin{align*}
	H_{(k)}u^{X_\perp h}_{\sigma_k} - h = u^{X_\perp (W^\star h)} + w^h_\psi,
    \end{align*}
    which upon taking the $k$-th Fourier component yields \eqref{eq:hkrc}. Theorem \ref{thm:main} is proved.     
\end{proof}

\subsection{Study of the operator $W_k$}\label{sec:propWk}

For $h\in \Omega_1$ non-vanishing smooth section (pick $e^{i\theta}$ in isothermal coordinates), the space $H_k$ is the completion of $h^k \C_0^\infty(M)$ for the $L^2(SM)$ topology. Recall that we define the space $\Omega_k := H_k\cap \C^\infty(SM)$. 

\begin{proposition}\label{prop:Wk}
    Assume the manifold $(M,g)$ is simple and let $h\in \Omega_1$ be non-vanishing. Then for any integer $k$, the operator $W_k: h^k\C_0^\infty(M)\to \Omega_k$ defined by 
    \begin{align}
	W_k f = (X_\perp u^f)_k,
	\label{eq:Wkdef}
    \end{align}
    can be extended to a smoothing operator $W_k:H_k\to \Omega_k$.
\end{proposition}

The proof of Prop. \ref{prop:Wk} requires making explicit the kernel of the operator $W_k$, which in turn requires introducing the following two transverse Jacobi fields $X_\perp \gamma = a(\x,\theta,t) \dot\gamma^\perp$ and $\partial_\theta \gamma = b(\x,\theta,t) \dot\gamma^\perp$, where $a,b$ are two scalar functions defined on $\D$ (defined in \eqref{eq:D}), solving the following differential equation on each geodesic
\begin{align*}
    \ddot a + \kappa(\gamma_{\x,\theta}(t)) a = 0 , \qquad a(\x,\theta,0) &= 1, \quad \dot a(\x,\theta,0) = 0, \\
    \ddot b + \kappa(\gamma_{\x,\theta}(t)) b = 0 , \qquad b(\x,\theta,0) &= 0, \quad \dot b(\x,\theta,0) = 1.
\end{align*}
The simplicity of the metric is equivalent to the fact that $b$ never vanishes outside $t=0$. In particular, we have the 
\begin{lemma}\label{lem:changevar}
    If $(M,g)$ is simple and $f$ is a smooth function on $M$, then the following change of variable formula holds for any $x\in M$:
    \begin{align}
	\int_M f(\y)\ dM_\y = \int_0^{2\pi} \int_0^{\tau(\x,\theta)} f(\gamma_{\x,\theta}(t)) b (\x,\theta,t)\ dt\ d\theta.	
	\label{eq:changevar}
    \end{align}
\end{lemma}

\begin{proof}[Proof of Lemma \ref{lem:changevar}] Let $\x\in M$. It is well-known that if the metric is simple, the mapping $(t,\theta)\mapsto \exp_\x(\theta,t) = \gamma_{\x,\theta}(t)$ is a global diffeomorphism mapping $\{(t,\theta): \theta\in [0,2\pi), t\in [0,\tau(\x,\theta))\}$ onto $M$, of which we now compute the jacobian:
    \begin{align*}
	\frac{\partial^2 \exp_\x(t,\theta)}{\partial t \partial\theta} = \det ( \partial_t \gamma_{\x,\theta}(t), \partial_\theta \gamma_{\x,\theta}(t)) = \det ( \dot\gamma, b \dot\gamma^\perp ) = b |\dot\gamma|^2 = b(\x,\theta,t) e^{-2\lambda(\gamma_{\x,\theta}(t))},
    \end{align*}
    where $|\cdot|$ denotes the Euclidean norm so that, since $\gamma$ is unit speed, $e^{2\lambda(\gamma_{\x,\theta}(t))} |\dot\gamma_{\x,\theta}(t)|^2 = 1$. We thus deduce that, using the change of variable $\y = \exp_\x (t,\theta) = \gamma_{\x,\theta}(t)$
    \begin{align*}
	dM_\y = e^{2\lambda(\y)} d^2 \y = e^{2\lambda(\gamma_{\x,\theta}(t))} b(\x,\theta,t) e^{-2\lambda(\gamma_{\x,\theta}(t))}\ dt\ d\theta = b(\x,\theta,t)\ dt\ d\theta,
    \end{align*}
    whence \eqref{eq:changevar}.    
\end{proof}

We now proceed to the proof of Prop. \ref{prop:Wk}.

\begin{proof}[Proof of Proposition \ref{prop:Wk}]
    Let us look at this operator in isothermal coordinates, i.e. assume that $g(\x) = e^{2\lambda(\x)} Id$ and an element $f\in \Omega_k$ can be written as $f(\x,\theta) = \tf(\x) e^{ik\theta}$. We may therefore study the operator $\wtW_k : \C_0^\infty (M) \to \C^\infty (M)$ defined by
    \begin{align*}
	\wtW_k (\tf) = e^{-ik\theta} W_k (\tf e^{ik\theta}), \qquad \tf\in\C_0^\infty(M),
    \end{align*}
    and show that it extends to an operator $\wtW_k:L^2(M)\to \C^\infty(M)$. By definition, 
    \begin{align*}
	\wtW_k (\tf)(\x) = \frac{1}{2\pi} \int_{\Sm^1} e^{-ik\theta} X_\perp \int_0^{\tau(\x,\theta)} \tf (\gamma_{\x,\theta}(t)) e^{ik \alpha_{\x,\theta}(t)}\ dt, 
    \end{align*}
    where we have defined $\alpha_{\x,\theta}(t) := \arg \dot\gamma_{\x,\theta}(t)$. We compute
    \begin{align*}
	2\pi\wtW_k (\tf)(\x) &= \int_{\Sm^1} e^{-ik\theta} X_\perp \int_0^{\tau(\x,\theta)} \tf (\gamma_{\x,\theta}(t)) e^{ik \alpha_{\x,\theta}(t)}\ dt\ d\theta, \\
	&= \int_{\Sm^1} e^{-ik\theta} \int_0^{\tau(\x,\theta)} X_\perp \left( \tf (\gamma_{\x,\theta}(t)) e^{ik \alpha_{\x,\theta}(t)}  \right)\ dt\ d\theta, \\
	&= \int_{\Sm^1} \int_0^{\tau(\x,\theta)} \left( \frac{a}{b} \partial_\theta (f\circ\gamma) + ik (f\circ\gamma) X_\perp \alpha  \right) e^{ik (\alpha_{\x,\theta}(t)-\theta)} \ dt\ d\theta.
    \end{align*}
    Integrating by parts the first term, we arrive at
    \begin{align*}
	\wtW_k (\tf)(\x) = \frac{1}{2\pi} \int_{\Sm^1} \int_0^{\tau(\x,\theta)} q_k (\x,\theta,t) \tf(\gamma_{\x,\theta}(t))\ dt\ d\theta,
    \end{align*}
    where we have defined
    \begin{align}
	q_k (\x,\theta,t) := \left( - \partial_\theta \left( \frac{a}{b} \right) - ik \frac{a}{b} (\partial_\theta \alpha_{\x,\theta}(t) - 1)  + ik X_\perp \alpha_{\x,\theta}(t)  \right) e^{ik (\alpha_{\x,\theta}(t) - \theta)}.
	\label{eq:defqk}
    \end{align}
    In order to simplify $q_k$, we need the following result, whose proof is given at the end of the present proof.
    \begin{lemma}\label{lem:alpha}
	We have, for any $(\x,\theta,t)\in \D$,
	\begin{align*}
	    X_\perp \alpha = \dot{a} - a \partial_t \lambda_\gamma, \qandq \partial_\theta \alpha = \dot{b} - b \partial_t \lambda_\gamma, \quad \text{where} \qquad \lambda_\gamma:= \lambda\circ\gamma.
	\end{align*}
    \end{lemma}
    Using Lemma \ref{lem:alpha}, we can establish that 
    \begin{align*}
	X_\perp \alpha - \frac{a}{b} (\partial_\theta\alpha-1) = \dot{a} - a \partial_t \lambda_\gamma - \frac{a}{b} (\dot{b} - b \partial_t \lambda_\gamma - 1) = (\dot{a} b - \dot{b}a + a)/b = (a-1)/b,
    \end{align*}
    where we have used the fact that $a \dot{b} - b\dot{a} = 1$. Then $q_k(\x,\theta,t)$ simplifies into
    \begin{align}
	q_k (\x,\theta,t) = \left( - \partial_\theta \left( \frac{a}{b} \right) + ik \frac{a-1}{b} \right) e^{ik (\alpha_{\x,\theta}(t) - \theta)}.
	\label{eq:qk}
    \end{align}
    It is established in \cite{Pestov2004} that $\partial_\theta \left( \frac{a}{b} \right)$ has a zero of order 1 at $t=0$ (this is partly what makes $W_0$ a smoothing operator). Further, it is clear that $a-1$ has a zero of order 2 while $b$ has a zero of order 1, so $(a-1)/b$ has a zero of order 1 at $t=0$ as well. Since $b(\x,\theta,t)$ only vanishes at first order at $t=0$ and is smooth on $\D$, the singularity at $t=0$ of the function $\frac{q_k(\x,\theta,t)}{b(\x,\theta,t)}$ is removable, making it a smooth function on $\D$. Therefore, using Lemma \ref{lem:changevar} and denoting $(\theta_\x(\y), t_\x(\y))$ the inverse mapping to the exponential map $(\theta,t)\mapsto \y= \gamma_{\x,\theta}(t)$, we arrive at 
    \begin{align*}
	\wtW_k (\tf) (\x) = \int_M \tf (\y) W_k(\x,\y) \ dM_\y, \where \qquad W_k(\x,\y) := \frac{q_k(\x,\theta_\x(\y), t_\x(\y))}{b(\x,\theta_\x(\y), t_\x(\y))},
    \end{align*}
    is a smooth kernel. The proof is complete.
\end{proof}

\begin{proof}[Proof of lemma \ref{lem:alpha}]
    We need to compute the quantities $X_\perp \alpha_{x,\theta}(t)$ and $\partial_\theta \alpha_{x,\theta}(t)$.
    Let us denote $\Gamma(s,t)$ a variation through geodesics with $\Gamma(0,t) = \gamma(t)$ and longitudinal and transverse fields $T(s,t) = \partial_t\Gamma$ and $S(s,t) = \partial_s\Gamma$ respectively equal to 
    \begin{align*}
	\partial_t \Gamma(s,t) = e^{-2\lambda(\Gamma(s,t))} \hat\alpha(s,t), \quad \hat\alpha := \binom{\cos\alpha}{\sin\alpha}, 
    \end{align*}
    with $T(0,t) = \dot\gamma(t)$ and $S(0,t)$ equal to the variation field (either $X_\perp \gamma$ or $\partial_\theta \gamma$ in our case). The quantity we aim at computing is thus $S[\alpha]|_{s=0}$ (we denote $Y[f] = Y^i \partial_i f$). We compute 
    \begin{align*}
	D_s \partial_t \Gamma &= \partial_s \left( e^{-2\lambda\circ\Gamma} \right) \hat\alpha(s,t) + e^{-2\lambda\circ\Gamma} D_s \hat\alpha \\ 
	&= \partial_s \left( e^{-2\lambda\circ\Gamma} \right) \hat\alpha(s,t) + e^{-2\lambda\circ\Gamma} \left( S(\alpha) \hat\alpha^\perp + \hat\alpha[\lambda]\ S - \hat\alpha^\perp[\lambda]\ S^\perp \right).
    \end{align*}
    By virtue of the symmetry lemma (see e.g. \cite[Lemma 6.3]{Lee1997}), note also that $D_s \partial_t\Gamma = D_t \partial_s \Gamma$. Using this fact and taking the inner product with $(\partial_t\Gamma)^\perp$, we arrive at the relation
    \begin{align*}
	S(\alpha) = \dprod{D_t \partial_s \Gamma}{\partial_t\Gamma^\perp} - \partial_t\Gamma[\lambda] \dprod{S}{\partial_t\Gamma^\perp} + \partial_t\Gamma^\perp [\lambda] \dprod{S^\perp}{\partial_t\Gamma^\perp}.
    \end{align*}
    We now set $s=0$ and, considering for instance the Jacobi field $X_\perp \gamma = a \dot\gamma^\perp$, we get
    \begin{align*}
	X_\perp \alpha = \dprod{D_t (a \dot\gamma_\perp)}{\dot\gamma^\perp} - \partial_t (\lambda\circ\gamma) \dprod{a \dot\gamma^\perp}{\dot\gamma^\perp} + \dot\gamma^\perp [\lambda] \dprod{-a\dot\gamma}{\dot\gamma^\perp} = \dot{a} - a \partial_t(\lambda\circ\gamma). 
    \end{align*}
    Similarly, considering the Jacobi field $\partial_\theta \gamma = b\dot\gamma^\perp$, we arrive at $\partial_\theta \alpha = \dot{b} - b \partial_t(\lambda\circ\gamma)$.
\end{proof}

Next we establish that $W_k$ and $W_k^\star$ are adjoint for the $L^2(SM)$ inner product structure on $H_k$. For that we will need the following

\begin{lemma}[Lemma 5.2 in \cite{Pestov2004}]\label{lem:oddeven}
    For any even function $f\in L^2(SM)$ and odd function $h\in L^2(SM)$, 
    \begin{align*}
	(If,Ih)_{L^2_\mu(\partial_+ SM)} = 0.
    \end{align*}
\end{lemma}

\begin{proposition}\label{prop:adjoint_k}
    $W_k^\star$ is the adjoint of $W_k$ in $H_k \subset L^2(SM)$.
\end{proposition}
The proof is very similar to the case $k=0$ proved in \cite[Proposition 5.3]{Pestov2004}. 

\begin{proof}
    It is enough to establish that 
    \begin{align*}
	(f,W_k^\star h)_{L^2(SM)} = (W_k f, h)_{L^2(SM)},
    \end{align*}
    for $f(x,\theta) = \tilde{f}(x) e^{ik\theta}$ and $h(x,\theta) = \tilde{h}(x) e^{ik\theta}$ for some $\tilde f, \tilde h \in \C_0^\infty(M)$. Doing so, we write
    \begin{align*}
	(f, W_k^\star h)_{L^2(SM)} &= \int_{SM} f (u^{X_\perp h})_k\ d\Sigma^3 = \int_{SM} f u^{X_\perp h}\ d\Sigma^3 = - \int_{SM} X(u^f) u^{X_\perp h}\ d\Sigma^3 \\
	&= - (If, I(X_\perp h))_{L^2_\mu(\partial_+ SM)} + \int_{SM} u^f X(u^{X_\perp h})\ d\Sigma^3.
    \end{align*}
    The first term in the right-hand side is zero by virtue of Lemma \ref{lem:oddeven} and working on the second term, we obtain
    \begin{align*}
	(f, W_k^\star h)_{L^2(SM)} = - \int_{SM} u^f X_\perp h\ d\Sigma^3 = \int_{SM} (X_\perp u^f) h\ d\Sigma^3 = \int_{SM} (X_\perp u^f)_k h\ d\Sigma^3,
    \end{align*}
    where the integration by parts has no boundary term since $h|_{\partial SM} = 0$. The proof is complete.
\end{proof}

\paragraph{Constant curvature cases.}
In the case of constant curvature, the functions $a,b$ do not depend on $x$ nor $\theta$ (in particular $\partial_\theta (a/b) = 0$), and they take the following values
\begin{align*}
    \kappa = 0 &: \quad a(t) \equiv 1, \quad b(t) = t, \\ 
    \kappa = -p^2 &:\quad a(t) = \cosh (pt), \quad b(t) = \sinh (pt), \\
    \kappa = p^2 &:\quad a(t) = \cos (pt), \quad b(t) = \sin(pt), 
\end{align*}
hence the respective kernels, obtained after using trigonometric identities:
\begin{align*}
    \kappa = 0 &: \quad q_k /b = 0, \\ 
    \kappa = -p^2 &:\quad q_k /b = ik (1+\cosh(pt))^{-1} e^{ik(\alpha_{\x,\theta}(t)-\theta)}, \\
    \kappa = p^2 &:\quad q_k /b = -ik (1 + \cos(pt))^{-1} e^{ik(\alpha_{\x,\theta}(t)-\theta)}. 
\end{align*}
In the Euclidean case $\kappa=0$, reconstruction formulas \eqref{eq:fkrc} and \eqref{eq:hkrc} are {\em exact}. However for $k\ne 0$, the operators $W_k$ and $W_k^\star$ no longer vanish in the nonzero constant curvature case. 

We now establish that the operators $W_k$ become contractions in $\L(H_k)$ norm whenever the metric is $\C^3$-close to Euclidean. This generalizes a result by Krishnan in \cite{Krishnan2010} for the case $k=0$ and metrics close to constant curvature. 

\begin{proposition}\label{prop:Wkest} Let $(M,g)$ a simple Riemannian manifold with the corresponding family of operators $W_k$ as defined above. Then there exist two constants $C_1,C_2>0$ such that for any integer $k$, the following estimate holds
    \begin{align}
	\|W_k\|_{\L(H_k)} \le C_1 \|\nabla \kappa\|_\infty + |k|\ C_2 \|\kappa\|_\infty.
	\label{eq:Wkest}
    \end{align}
\end{proposition}
 
\begin{proof} Since $(M,g)$ is simple, note that applying Lemma \ref{lem:changevar} to $f^2$, we have that for every $\x\in M$,
    \begin{align}
	\int_M f^2(\y) \ dM_\y = \int_{0}^{2\pi} \int_0^{\tau(\x,\theta)} f^2(\gamma_{\x,\theta}(t)) b (\x,\theta,t)\ dt\ d\theta
	\label{eq:L2id}
    \end{align}
    Therefore, we may bound formally
    \begin{align*}
	2\pi|\widetilde{W_k}(\tilde f)(\x)| &= \left| \int_{0}^{2\pi}\int_0^{\tau(\x,\theta)} q_k(\x,\theta,t) \tilde{f}(\gamma_{\x,\theta}(t))\ dt\ d\theta \right| \\
	&= \left| \int_{0}^{2\pi} \int_0^{\tau(\x,\theta)} \left( \frac{q_k}{\sqrt{b}} \right)\left( \tilde{f}_\gamma \sqrt{b} \right)\ dt\ d\theta \right | \\
	&\le \sqrt{\int_{0}^{2\pi} \int_0^{\tau(\x,\theta)} \frac{|q_k|^2}{b^2} b\ dt\ d\theta } \sqrt{\int_{0}^{2\pi} \int_0^{\tau(\x,\theta)}  \tilde{f}^2_\gamma\ b\ dt\ d\theta}.
    \end{align*}
    Squaring, integrating over $\x\in M$ and using \eqref{eq:L2id}, we arrive at the relation
    \begin{align*}
	\frac{\|\widetilde{W}_k(\tilde f)\|_{L^2(M)}^2}{\|\tilde{f}\|_{L^2(M)}^2} \le \frac{1}{(2\pi)^2} \int_M \int_{0}^{2\pi} \int_0^{\tau(\x,\theta)} \frac{|q_k|^2}{b^2}(\x,\theta,t) b(\x,\theta,t)\ dt\ d\theta\ dM_\x . 
    \end{align*}
    Given expression \eqref{eq:qk}, we have 
    \begin{align*}
	\frac{|q_k|^2}{b^2} = \left(\frac{1}{b} \partial_\theta \frac{a}{b} \right)^2 + k^2 \left( \frac{a-1}{b^2} \right)^2.
    \end{align*}
    By writing an ODE in $t$ for the function $b\partial_\theta a - a\partial_\theta b$, it is proved in \cite{Krishnan2010} that the function $|\frac{1}{t}\partial_\theta(a/b)|$ is uniformly bounded by $C \|\nabla\kappa\|_\infty$ for some positive constant $C$. Since $\frac{b}{t}$ is smooth over the compact domain $\overline{\D}$ (defined in \eqref{eq:D}), non-vanishing by virtue of simplicity, and with $\lim_{t\to 0} \frac{b}{t} = 1$, we may then deduce that $|\frac{1}{b}\partial_\theta(a/b)| \le C \|\nabla \kappa\|_\infty$ for some constant $C$, which in turn proves the first term in the right-hand side of \eqref{eq:Wkest}. Although it would be enough to work with $L^2$ estimates, we will complete the proof by showing a uniform estimate of the form 
    \begin{align}
	\left| \frac{a-1}{b^2}\right| \le C' \|\kappa\|_\infty, \quad (x,\theta,t) \in \D.
	\label{eq:estam1b}
    \end{align}
    
    {\bf Bound on $\frac{a-1}{b^2}$ (proof of \eqref{eq:estam1b}):} The function $a-1$ satisfies the ODE 
    \begin{align*}
	\ddot{ \overline{a-1}} + \kappa_\gamma (a-1) = -\kappa_\gamma , \quad (a-1)(0) = \dot{\overline{(a-1)}}(0) = 0, 
    \end{align*}
    so if $\Phi = \left[\begin{smallmatrix} a & b \\ \dot{a} & \dot{b} \end{smallmatrix}\right]$ is the fundamental matrix solution of $\ddot\Phi + \kappa_\gamma \Phi = 0$ with $\Phi(0) = \left[\begin{smallmatrix} 1 & 0 \\ 0 & 1 \end{smallmatrix}\right]$ (note that $\det \Phi(t) = 1$ for every $t\ge 0$), we have that 
    \begin{align*}
	(a-1)(t) = \left[
	\begin{array}{cc}
	    1 & 0
	\end{array}
    \right] \Phi(t) \int_0^t \Phi(s)^{-1} \left[
	    \begin{array}{c}
		0 \\ -\kappa_\gamma(s) 
	    \end{array}
	\right] \ ds = \int_0^t (a(t)b(s) - a(s) b(t)) \kappa_\gamma(s)\ ds, 
    \end{align*}
    so that we have
    \begin{align*}
	\frac{a-1}{b^2}(t) = \frac{t}{b(t)}\int_0^1 \frac{a(t)b(tu) - a(tu)b(t)}{b(t)} \kappa_\gamma(tu)\ du,
    \end{align*}
    where the function $F_{\x,\theta}(t,u) := \frac{a(t)b(tu) - a(tu)b(t)}{b(t)} \to u-1$ as $t\to 0$ and smooth w.r.t. all its arguments, thus uniformly bounded on $\overline{\D}\times [0,1]$. With $t/b(t)$ bounded over $\D$ as well, we deduce  estimate \eqref{eq:estam1b}, and the proof is complete.    
\end{proof}

As an obvious corollary, we have
\begin{corollary}\label{cor:contraction}
    For any fixed $k\in \Zm$ there exists a constant $C_k$ such that 
    \begin{align}
	\|\kappa\|_\infty + \|\nabla\kappa\|_\infty \le C_k,
	\label{eq:curvcond}
    \end{align}
    implies that the operators $W_k, W_k$ are contractions in $H_k$. In this case, the operators $I + W_k^2$ and $I+ W_k^{\star 2}$ are invertible elements of $\L(H_k)$. 
\end{corollary}
For metrics with curvature satisfying \eqref{eq:curvcond}, the operators $I+W_k^2$ and $I+(W_k^\star)^2$ become invertible via Neumann series, and equations \eqref{eq:fkrc} and \eqref{eq:hkrc} become {\em exact} reconstruction formulas
\begin{align}
    f = - \sum_{p=0}^\infty (-W_k^2)^p (X_\perp w_\psi^{(f)})_k, \qandq  h = - \sum_{p=0}^\infty (-W_k^{\star2})^p (w_\psi^{(h)})_k. \label{eq:neumanns}
\end{align}

\section{Numerical simulations} \label{sec:numerics}

\subsection{Introduction}

The purpose of this section is twofold:
\begin{itemize}
    \item[$(i)$] Validate the exactness of inversion formulas \eqref{eq:neumanns} in some cases of simple and close-to-simple manifolds. 
    \item[$(ii)$] Illustrate the fact that in the case of constant curvature, unlike the case $k=0$ where the error operators $W_0, W_0^\star$ vanish identically, they no longer do for $k\ne 0$ and if curvature is too large, the Neumann series can diverge. 
\end{itemize}

Least-squares methods for the reconstruction of vector fields from knowledge of their ray transform and additional information (their transverse ray transform) may be found in \cite{Svetov2013}. Reconstruction of solenoidal tensor fields using bases of solenoidal polynomials and a least-square method was investigated in \cite{Derevtsov2005}. Here we use the code recently developped in {\tt MatLab} by the author and presented in \cite{Monard2013}, based on discretizing explicitely the integral and differential operators involved. 
\medskip

\noindent{\bf Formulation and approach.} In both problems presented, the unknown tensor always depends on one {\em scalar function} $f$, as it is either of the form $f(x) e^{ik\theta}$ or $X_\perp (f(x) e^{ik\theta})$. For visualization purposes, this is what we reconstruct numerically. It is customary to introduce the notation $I_k f := I [f e^{ik\theta}]$ and $I_{k,\perp} f = I [X_\perp (f e^{ik\theta})]$, and the problems we tackle numerically, equivalent to the ones treated in the sections above, is to reconstruct the function $f$ from either data $I_k f$ or $I_{k,\perp} f$, with $k$ some fixed integer determining the order of the initial tensor. If we introduce the operators 
\begin{align*}
    \wtW_k (f) = e^{-ik\theta} W_k (f e^{ik\theta}) \qandq \wtW_k^\star (f) = e^{-ik\theta} W_k^\star (f e^{ik\theta}), \quad f\in L^2(M),
\end{align*}
then both operators are $L^2(M)$-adjoint of one another and smoothing when the metric is simple. The series in \eqref{eq:neumanns} can now be written in the sense of {\em functions} (and not elements of $\Omega_k$) as
\begin{align}
f &= - \sum_{p=0}^\infty (-\wtW_k^2)^p \widetilde{(X_\perp w_\psi^{(f)})_{k}}, \quad w^{(f)} := H_{(k)} (I_k f)_{-\sigma_k} |_{\partial_+ SM} \label{eq:fhrc} \\
h &= - \sum_{p=0}^\infty (-\wtW_k^{\star2})^p \widetilde{(w_\psi^{(h)})_{k}}, \quad w^{(h)} := H_{(k)} (I_{k,\perp} h)_{\sigma_k}|_{\partial_+ SM}. \label{eq:hhrc}
\end{align}
Formally, both series \eqref{eq:fhrc} and \eqref{eq:hhrc} come from equations of the form
\begin{align*}
    f - \K f = AIf,
\end{align*}
with $f$ the unknown, $I$ the forward map, $A$ an approximate inversion (parametrix) and $\K$ an ``error'' operator. Since $\K$ can be expressed as $\K f = f - AIf$, whenever the equation above is invertible via a Neumann series, one may write 
\begin{align*}
    f = \sum_{p=0}^\infty \K^p AIf = \sum_{p=0}^\infty (Id - AI)^p AIf.    
\end{align*}
Therefore, such interative formulas only require implementing the forward map $I$ and the approximate reconstruction formula $A$ in order to be computed, and they do not require computing the error operators (for the case of formulas \eqref{eq:fhrc}, they are $-\widetilde{W_k}^2$ and its adjoint). Such an approach is systematic of explicit inversion formulas modulo contractive error operators, see e.g. \cite{Monard2013,Qian2011}.

Although it is not clear how to quantify the control over curvature (which ensures convergence of the series) in Corollary \ref{cor:contraction}, we will see that the series successfully converges for a one-parameter family of metrics which can reach the limit of simplicity. In the sections below, we only briefly recall the approach to focus on the qualitative numerical behavior. The interested reader may find a more detailed explanation of the method in \cite{Monard2013}.

Though the code developped in \cite{Monard2013} can handle more general boundaries, we restrict the computational domain to be the unit disc here. We will show the outcome of reconstruction algorithms inverting for a real-valued function $f(x,y)$ (the phantom visualized in Fig. \ref{fig:phantom2}) from knowledge of either $I_k f$ or $I_{k,\perp} f$. The metrics considered will be of constant negative and positive curvature in section \ref{ssec:constcurv}, then we will use in section \ref{ssec:oneparam} a one-parameter family of metrics which transitions continuously from simple to non-simple to trapping. 

\begin{figure}[htpb]
    \centering 
    \includegraphics[trim = 30 18 30 30, clip, width=0.228\textwidth]{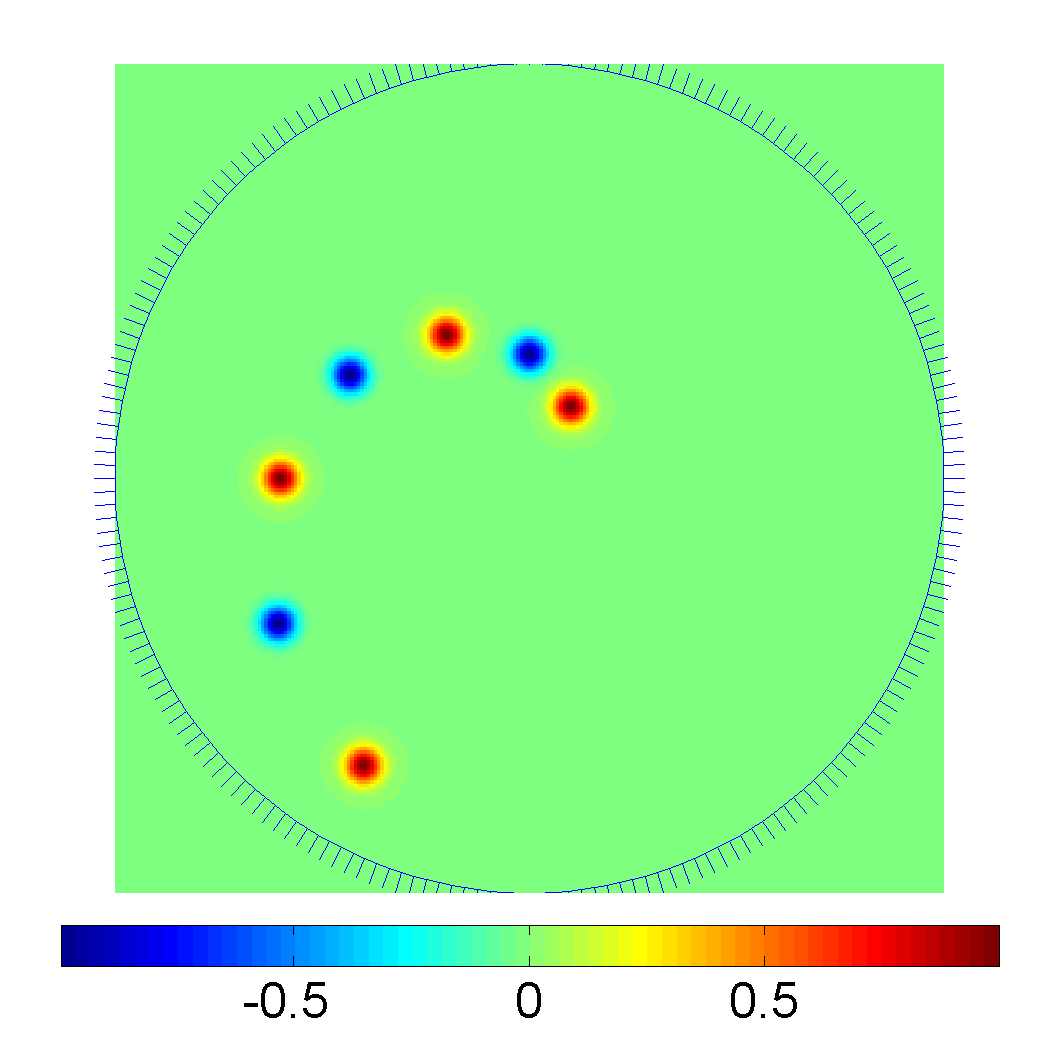}
    \caption{Phantom $f$ embedded into the computational domain (unit disc $\{x^2 + y^2 < 1\}$).}
    \label{fig:phantom2}
\end{figure}

\subsection{Implementation}

\subsubsection{Implementation of the forward operators}
The forward operators $f \mapsto I_k f:= I(f e^{ik\theta})$ and $f\mapsto I_{k,\perp} f := I(X_\perp (f e^{ik\theta}))$ are defined on the influx boundary $\partial_+ SM$ of the domain, parameterized by $(\beta,\alpha) \in [0,2\pi]\times \left[ -\frac{\pi}{2}, \frac{\pi}{2} \right]$, where $\beta$ parameterizes the position of the boundary point $\x(\beta) = (x(\beta),y(\beta))^T$ from where to shoot the geodesic, and $\alpha$ denotes the angle of initial speed direction, ranging from $\frac{-\pi}{2}$ to $\frac{\pi}{2}$ with respect to the inner normal at $\x(\beta)$. In the present situation where the domain is the unit disc, it coincides with the so-called {\em fan-beam coordinates}, where we have $\x(\beta) = \binom{\cos\beta}{\sin\beta}$ and the inner normal there has angle $\beta + \pi$ with respect to the $x$-axis. Computing the forward operators  consists of the following steps. 
\begin{enumerate}
    \item Discretizing $\partial_+ SM$ appropriately. Here we will choose $2n\times n$ equispaced points in $[0,2\pi]\times [-\frac{\pi}{2},\frac{\pi}{2}]$, where $n$ is the sidelength of the reconstruction grid.
    \item For each boundary point $(\beta,\alpha)$ in this discretization, compute the corresponding geodesic $\{ (\gamma_{\beta,\alpha},\dot{\gamma}_{\beta,\alpha})(t),\ 0\le t\le \tau(\beta,\alpha) \}$ by solving numerically the geodesic system 	
	\begin{align*}
	    \dot \x(t) = e^{-\lambda(\x(t))}\ \hat\theta(t), \qquad \dot\theta(t) = e^{-\lambda(\x(t))}\ \hat\theta(t)^\perp\cdot\nabla \lambda(\x(t)),
	\end{align*}
	with initial conditions
	\begin{align*}
	    \x(0) = \hat\beta, \quad \theta(0) = \beta + \pi + \alpha. 
	\end{align*}	
	This is done by marching forward in time with stepsize $\Delta t$, until the geodesic exits the domain. The exit test raises a flag whenever $x(t)^2 + y(t)^2 \ge 1$. The outcome of this procedure is a collection of points of the form $(\x_{\beta,\alpha}^p, \theta_{\beta,\alpha}^p)_{p=1}^N$. It is sufficient to take $N$ as an integer larger than $\diam(M)/\Delta t$. The metric and its partial derivatives are defined by analytic expressions so that there is no particular underlying Eulerian grid in the forward problem. 
    \item Using the computed geodesics, compute $I_k f$ and $I_{k,\perp} f$ by the following quadratures (for clarity, we make the $(\beta,\alpha)$-dependence of $(\x^p,\theta^p)_{p=1}^N$ implicit here)
	\begin{align*}
	    I_k f(\beta,\alpha) &\approx \Delta t\sum_{p=1}^N f(\x^p) e^{ik \theta^p}, \\ 
	    I_{k,\perp} f (\beta,\alpha) &\approx \Delta t\sum_{p=1}^N e^{-\lambda(\x^p)} e^{ik \theta^p} \left( \frac{f(\x^{p,+}) - f(\x^{p,-})}{2\Delta t} + ik\ \widehat{\theta^p}\cdot\nabla\lambda (\x^p) f(\x^p) \right),
	\end{align*}		
	where we have defined $\x^{p,\pm} := \x^p \mp \Delta t\ \widehat{\theta^p}^\perp$.
\end{enumerate}

\subsubsection{Implementation of the approximate inversions}

\paragraph{Right-hand-side of \eqref{eq:fkrc}.}
We first rewrite the right-hand side of \eqref{eq:fkrc} (denote $w\equiv w^{(f)}$), so that the differentiation step only occurs on the final cartesian grid. 
\begin{align}
    - \widetilde{(X_\perp w^f_\psi)_k} &= \frac{e^{-\lambda}}{2\pi} \int_{\Sm^1} e^{-ik\theta} ( \hat\theta^\perp \cdot\nabla w_\psi - (\hat\theta\cdot\nabla\lambda) \partial_\theta w_\psi )\ d\theta \nonumber\\
    &= \frac{e^{-\lambda}}{2\pi} \int_{\Sm^1} (\hat\theta^\perp\cdot\nabla w_\psi e^{-ik\theta} + \partial_\theta \left( e^{-ik\theta} \hat\theta\cdot\nabla \lambda \right) w_\psi)\ d\theta \nonumber\\
    &= \frac{e^{-2\lambda}}{2\pi} \nabla\cdot \int_{\Sm^1} e^{-ik\theta} e^\lambda \hat\theta^\perp w_\psi\ d\theta - \frac{ik e^{-\lambda}}{2\pi} \nabla\lambda\cdot \int_{\Sm^1} e^{-ik\theta} \hat\theta w_\psi\ d\theta. \label{eq:lastRHS}
\end{align}
At each point of the domain, the computation of \eqref{eq:lastRHS} consists of the following steps:
\begin{enumerate}
    \item Compute $w^{(f)} = (H_{(k)} I_k f)_{-\sigma_k}$. First extend the data to the range $\alpha\in[-\pi,\pi]$ by oddness or evenness depending on $\sigma_k$. Then compute the fiberwise shifted Hilbert transform $H_{(k)}$: this is done for each $\beta$-slice separately via Fast Fourier Transform. Finally, restrict it back to $\alpha\in [-\frac{\pi}{2},\frac{\pi}{2}]$.
    \item For each gridpoint, compute the integrals $u(\x) := \int_{\Sm^1} w^{(f)}_\psi(\x,\theta) e^{-ik\theta} \cos\theta\ d\theta$ and $v(\x) := \int_{\Sm^1} w^{(f)}_\psi(\x,\theta) e^{-ik\theta} \sin\theta \ d\theta$, where each access $w_\psi(\x,\theta)$ requires computing the basepoint $\alpha\circ\psi(\x,\theta)$ by following backwards the geodesic with initial conditions $(\x,\theta)$. 
    \item Compute $\frac{e^{-2\lambda}}{2\pi} (- \partial_x (e^\lambda v) + \partial_y (e^\lambda u)) - \frac{ik e^{-\lambda}}{2\pi} (u\partial_x \lambda + v\partial_y\lambda)$ at each point of the reconstruction grid using centered finite differences and pointwise multiplications. 
\end{enumerate}

\paragraph{Right-hand-side of \eqref{eq:hkrc}.}

Computing the right-hand-side of \eqref{eq:hkrc} requires fewer steps than the previous one:
\begin{enumerate}
    \item Compute $w^{(h)} = (H_{(k)} I_{k,\perp} f)_{\sigma_k}$. First extend the data to the range $\alpha\in[-\pi,\pi]$ by oddness or evenness depending on $\sigma_k$. Then compute the fiberwise Hilbert transform $H_{(k)}$, done as above and restrict the result back to $\alpha\in [-\frac{\pi}{2},\frac{\pi}{2}]$.
    \item For each gridpoint, compute $f(\x) = \frac{1}{2\pi} \int_{\Sm^1} w^{(h)}_\psi(\x,\theta) e^{-ik\theta}\ d\theta$. Again, each access $w^{(h)}_\psi(\x,\theta)$ requires computing the basepoint $\alpha\ \circ\ \psi(\x,\theta)$ by following backwards the geodesic with initial conditions $(\x,\theta)$. 
\end{enumerate}

\subsection{Numerical resuts}
\subsubsection{Metrics with constant curvature} \label{ssec:constcurv}

For $R > 1$, the isotropic metric
\begin{align}
    g_{R,\pm}(x,y) := \frac{4R^4}{(x^2 + y^2 \pm R^2)^2}, \qquad x^2 + y^2 \le 1,
    \label{eq:constCurvMetric}
\end{align}
has constant Gaussian curvature $\kappa = \pm \frac{1}{R^2}$ throughout the unit disc. As $R\to \infty$, both models of constant positive and negative curvature converge poinwise to the Euclidean metric. Some examples of geodesics for each case are given Fig. \ref{fig:CCgeo}.

\begin{figure}[htpb]
    \centering
    \includegraphics[trim = 26 40 28 40, clip, width=0.13\textwidth]{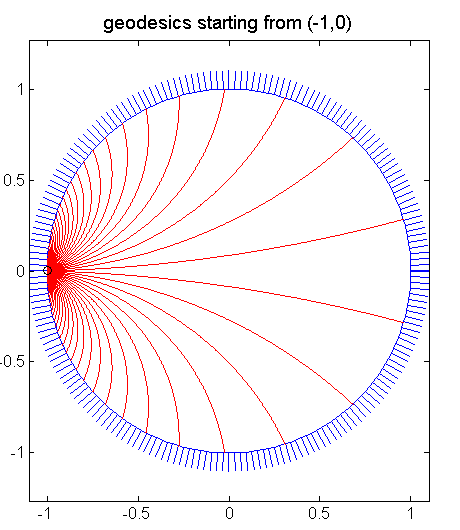}
    \includegraphics[trim = 26 40 28 40, clip, width=0.13\textwidth]{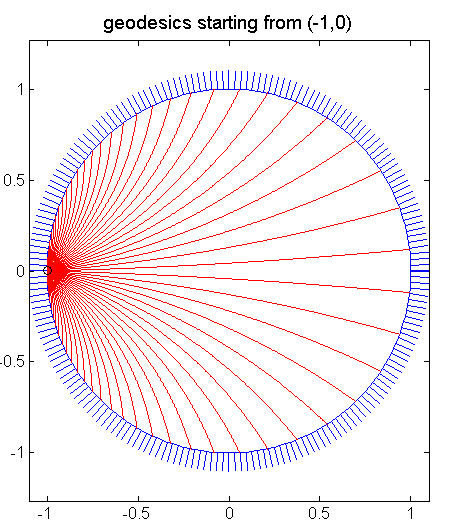}
    \includegraphics[trim = 26 40 28 40, clip, width=0.13\textwidth]{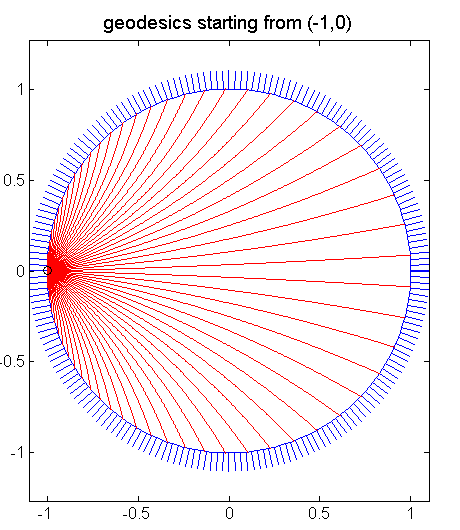}
    \includegraphics[trim = 26 40 28 40, clip, width=0.13\textwidth]{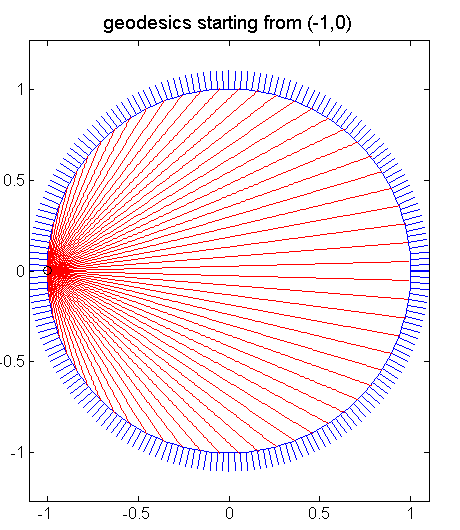}
    \includegraphics[trim = 26 40 28 40, clip, width=0.13\textwidth]{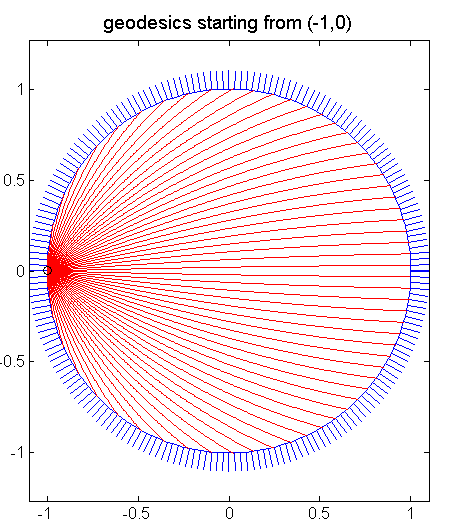}
    \includegraphics[trim = 26 40 28 40, clip, width=0.13\textwidth]{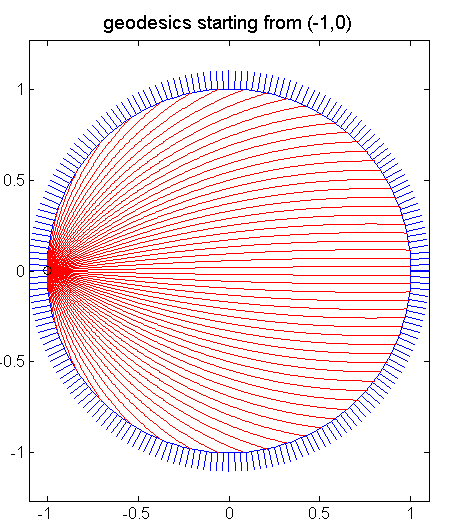}
    \includegraphics[trim = 26 40 28 40, clip, width=0.13\textwidth]{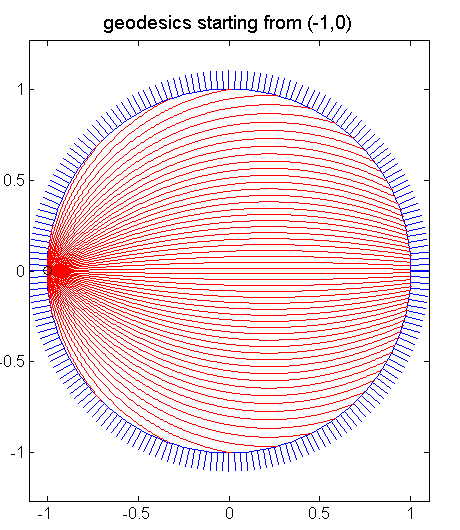}
    \caption{Some geodesics for the constant curvature cases. Left to right is arranged by increasing order of constant curvature $\kappa \in \{-0.69, -0.39, -0.25, 0.0, 0.25, 0.39, 0.69\}$, the middle one being Euclidean.}
    \label{fig:CCgeo}
\end{figure}

\begin{description}
    \item[Experiment 1: Constant positive curvature (CPC).]  For values of order $k\in \{3,6,10\}$ of the differential, and parameter values $R \in \{1.2, 1.6, 2.0\}$ for the metric $g_{R,+}$ corresponding to Gaussian curvatures $\kappa \in \{0.60, 0.39, 0.25\}$, we present the outcome of $10$ iterations of the series \eqref{eq:fhrc}. Note that $\kappa$ must be strictly less than $1$ (i.e. $R>1$), since in this limit case, the boundary is totally geodesic, which violates the strict convexity condition. The set of values chosen for $k$ includes at least one even and one odd value to check the validity of the reconstruction algorithm in each case. Relative $L^2$ errors after 10 iterations are summarized in Table \ref{tab:1}. 
    \item[Experiment 2: Constant negative curvature (CNC).] We repeat the same experiment with the same values of $m \in \{3, 6, 10\}$ and $R\in \{1.2, 1.6, 2.0\}$ as in Experiment 1, this time using the metric $g_{R,-}$ with corresponding Gaussian curvatures $\kappa \in \{-0.60, -0.39, -0.25\}$. Note that $\kappa$ must be strictly greater than $-1$ (i.e. $R>1$), since in this limit case, the manifold becomes the entire Poincar\'e disk and has no boundary. Relative $L^2$ errors after 10 iterations are summarized in Table \ref{tab:1}.  
\end{description}

\begin{figure}[htpb]
    \centering
    \subfigure[$I_3 f$ (mod/arg), metric $g_{R,+}$ with $R = 2.0$]{
    \includegraphics[trim = 10 20 40 30, clip, width=0.48\textwidth]{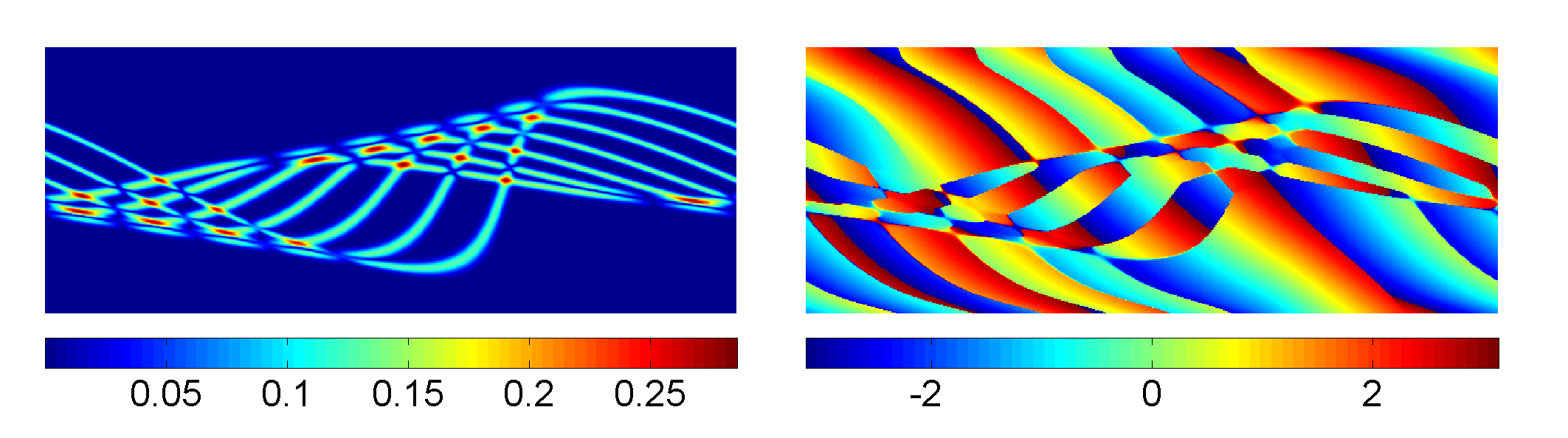}
    \label{fig:CPC_I3f_R2}
    }
    \subfigure[$I_3 f$ (mod/arg), metric $g_{R,-}$ with $R = 2.0$]{
    \includegraphics[trim = 10 20 40 30, clip, width=0.48\textwidth]{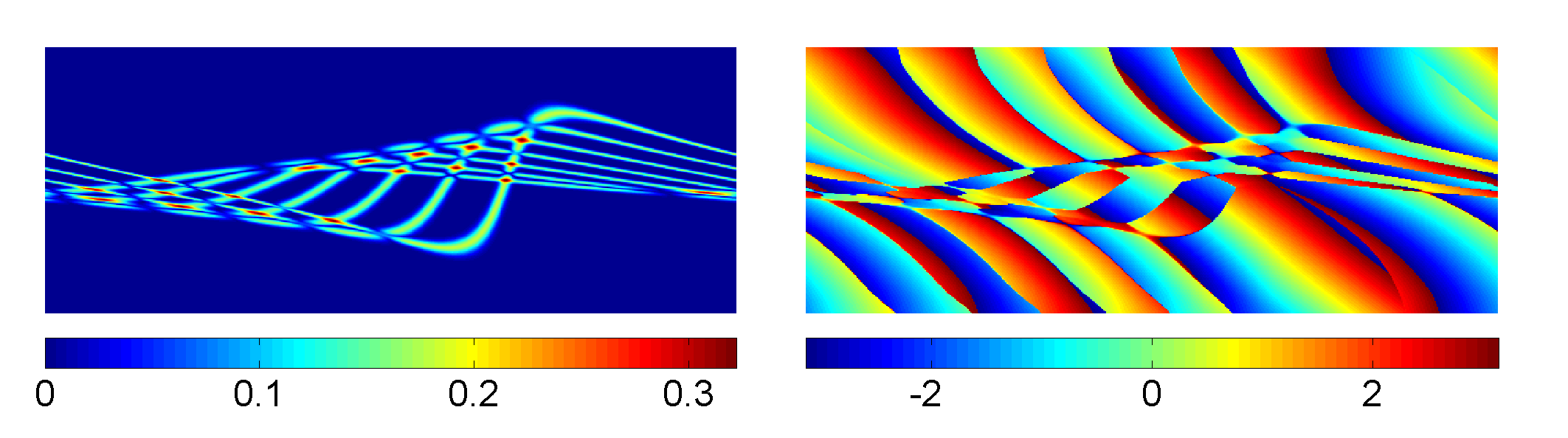}
    \label{fig:CNC_I3f_R2}
    }
    \subfigure[$I_6 f$ (mod/arg), metric $g_{R,+}$ with $R = 2.0$]{
    \includegraphics[trim = 10 20 40 30, clip, width=0.48\textwidth]{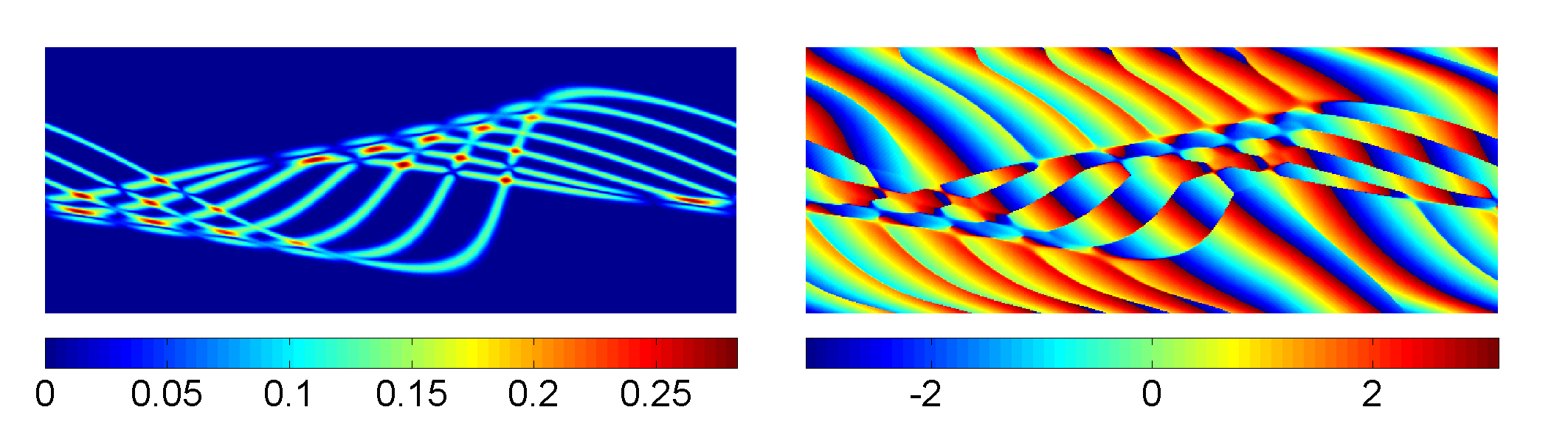}
    \label{fig:CPC_I6f_R2}
    }
    \subfigure[$I_6 f$ (mod/arg), metric $g_{R,-}$ with $R = 2.0$]{
    \includegraphics[trim = 10 20 40 30, clip, width=0.48\textwidth]{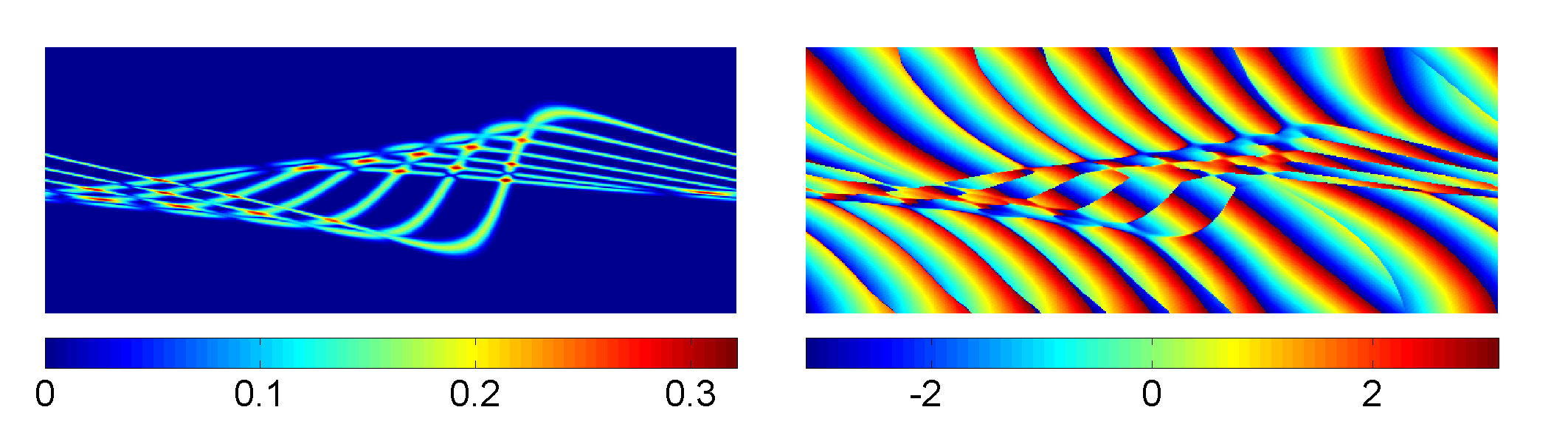}
    \label{fig:CNC_I6f_R2}
    }    
    \subfigure[$I_6 f$ (mod/arg), metric $g_{R,+}$ with $R = 1.2$]{
    \includegraphics[trim = 10 20 40 30, clip, width=0.48\textwidth]{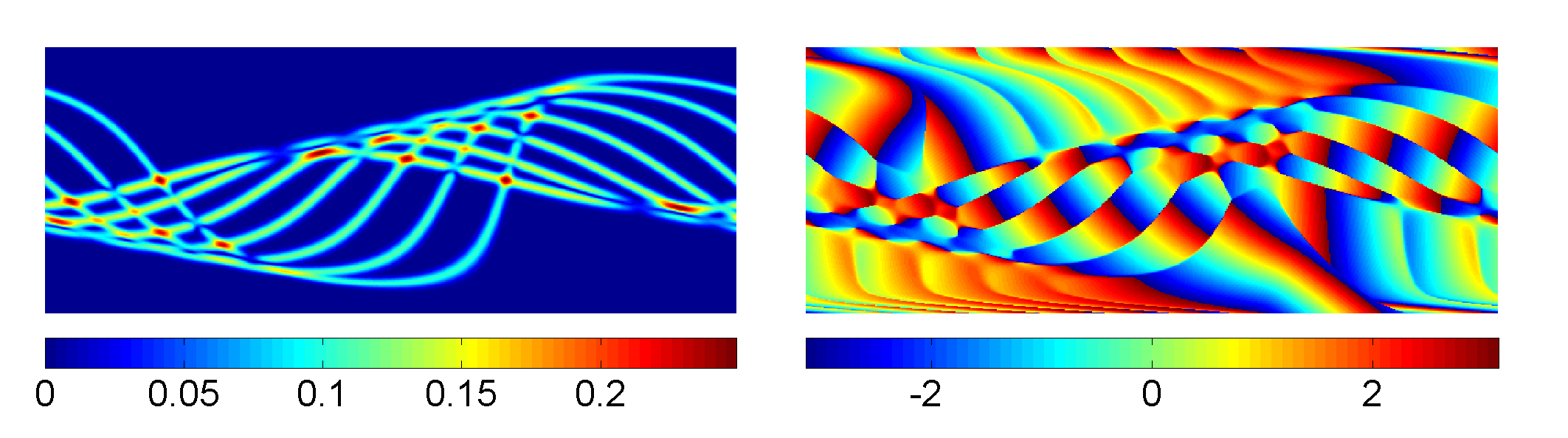}
    \label{fig:CPC_I6f_R1p2}
    }
    \subfigure[$I_6 f$ (mod/arg), metric $g_{R,-}$ with $R = 1.2$]{
    \includegraphics[trim = 10 20 40 30, clip, width=0.48\textwidth]{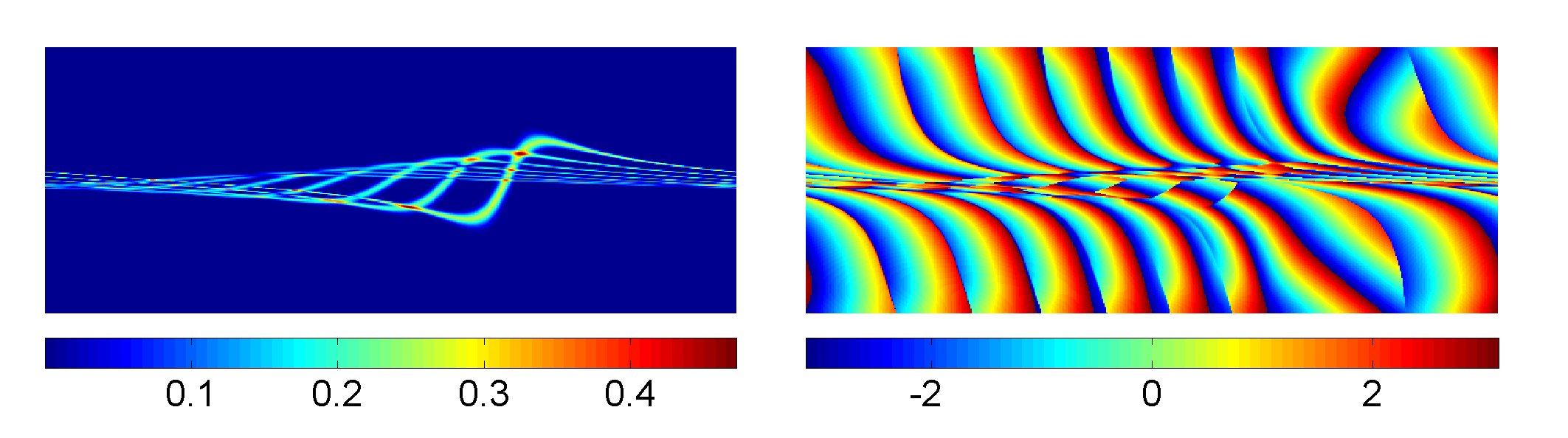}
    \label{fig:CNC_I6f_R1p2}
    }
    \caption{Examples of forward data $I_k f$ for various values of tensor order $k$ and metrics with constant curvature. The data is complex-valued and each plot contains modulus and argument. Axes are $(\beta,\alpha)\in [0,2\pi]\times [-\frac{\pi}{2}, \frac{\pi}{2}]$.}
    \label{fig:CCIf}
\end{figure}

\begin{table}
    \centering
    \begin{tabular}[htpb]{|c||c|c|c|}
	\hline
	CPC & $\kappa = 0.25$ & $\kappa = 0.39$ & $\kappa = 0.69$ \\ 
	\hline
	\hline
	$k = 3$ & 1.0 & 1.0  & 2.8 \\
	\hline
	$k = 6$ & 1.3 & 2.7 & 17.5 (NC) \\
	\hline
	$k = 10$ & 2.8 & 5.5 (NC) & 35.2 (NC) \\
	\hline
    \end{tabular}
    \quad     
    \begin{tabular}[htpb]{|c||c|c|c|}
	\hline
	CNC & $\kappa = -0.25$ & $\kappa = -0.39$ & $\kappa = -0.69$ \\ 
	\hline
	\hline
	$k = 3$ & 3.0 & 7.4 & 38.4 (DV) \\
	\hline
	$k = 6$ & 3.0  & 8.0 & 59 (DV) \\
	\hline
	$k = 10$ & 4.2 & 10 & 180 (DV) \\
	\hline
    \end{tabular}
    \caption{$L^2$ relative errors (in \%) after $10$ iterations for cases of constant curvature (left: positive, right:negative). $k$: tensor order. $\kappa$: curvature. (NC) indicates that the series had not reached convergence yet was still converging. (DV) indicates divergence. }
    \label{tab:1}
\end{table}

\begin{figure}[htpb]
    \centering
    \subfigure[CPC: relative $L^2$ error convergence plots, $k\in \{3,6,10\}$, $R\in \{1.2,1.6,2.0\}$]{
    \includegraphics[trim = 5 0 15 25, clip, height=0.32\textheight]{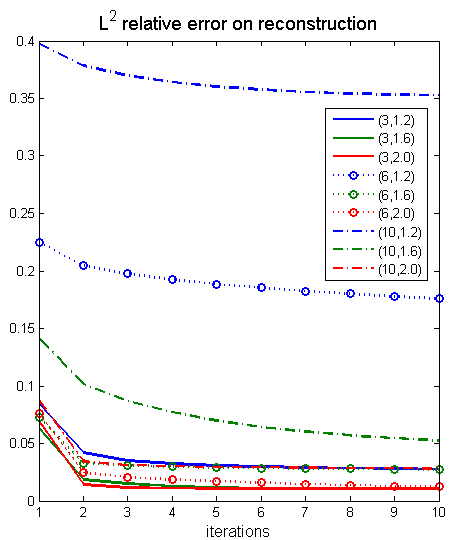}
    \label{fig:convplot1}
    }
    \subfigure[CNC: relative $L^2$ error convergence plots, $k\in \{3,6,10\}$, $R\in \{1.2,1.6,2.0\}$]{
    \includegraphics[trim = 5 0 15 25, clip, height=0.32\textheight]{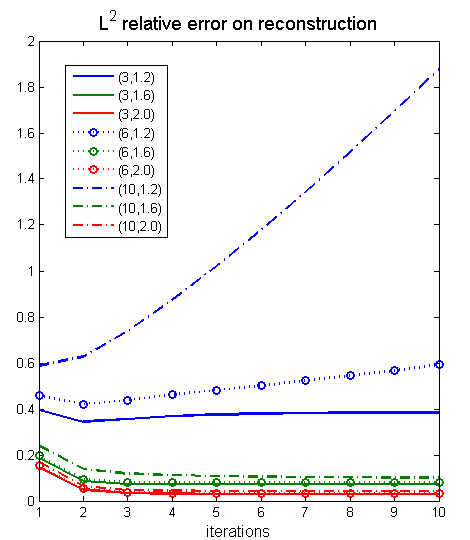}
    \label{fig:convplot2}
    }
    \subfigure[$k=10$, $\kappa = 0.69$]{
    \includegraphics[trim = 10 10 30 30, clip, width=0.4\textwidth]{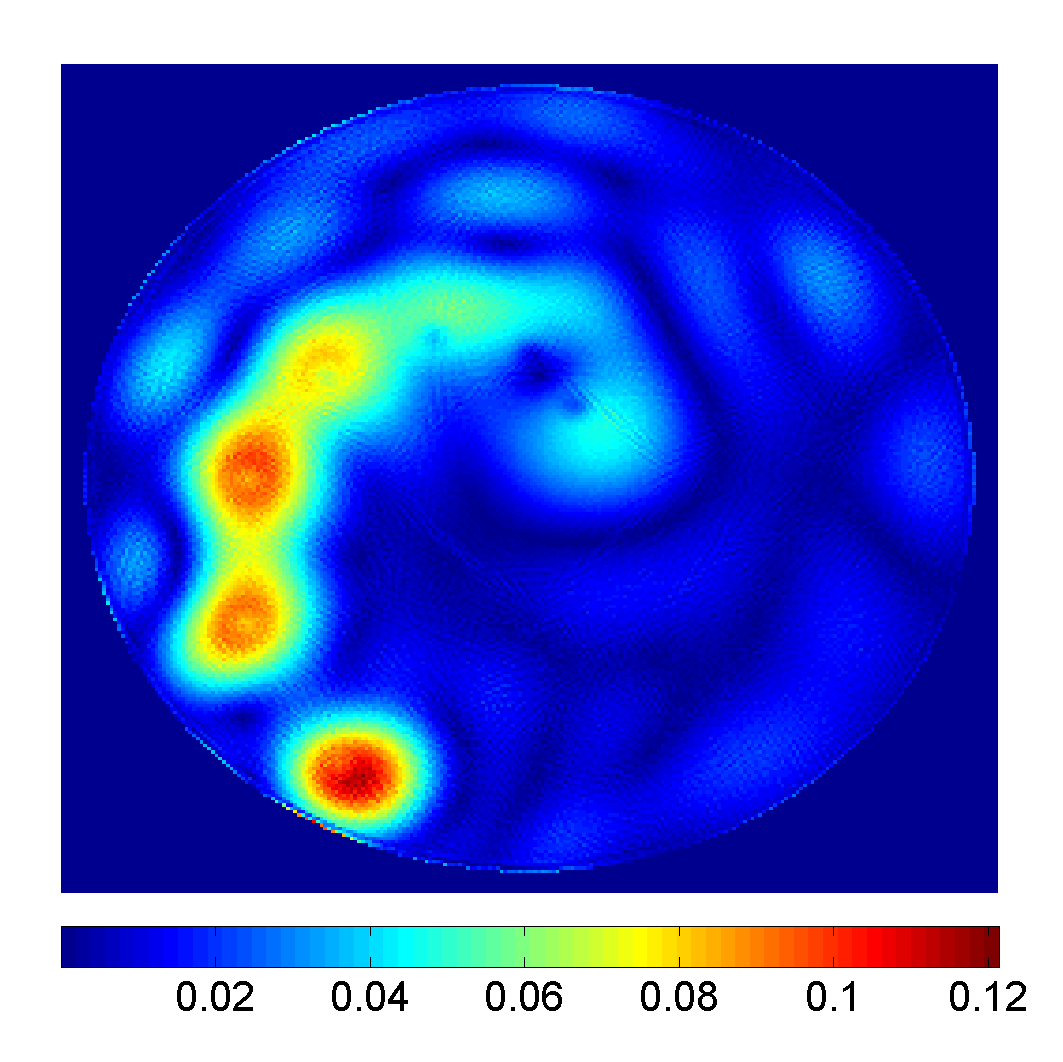}
    \label{fig:CPC_artifact}
    }
    \subfigure[$k=10$, $\kappa = -0.69$]{
    \includegraphics[trim = 10 10 30 30, clip, width=0.4\textwidth]{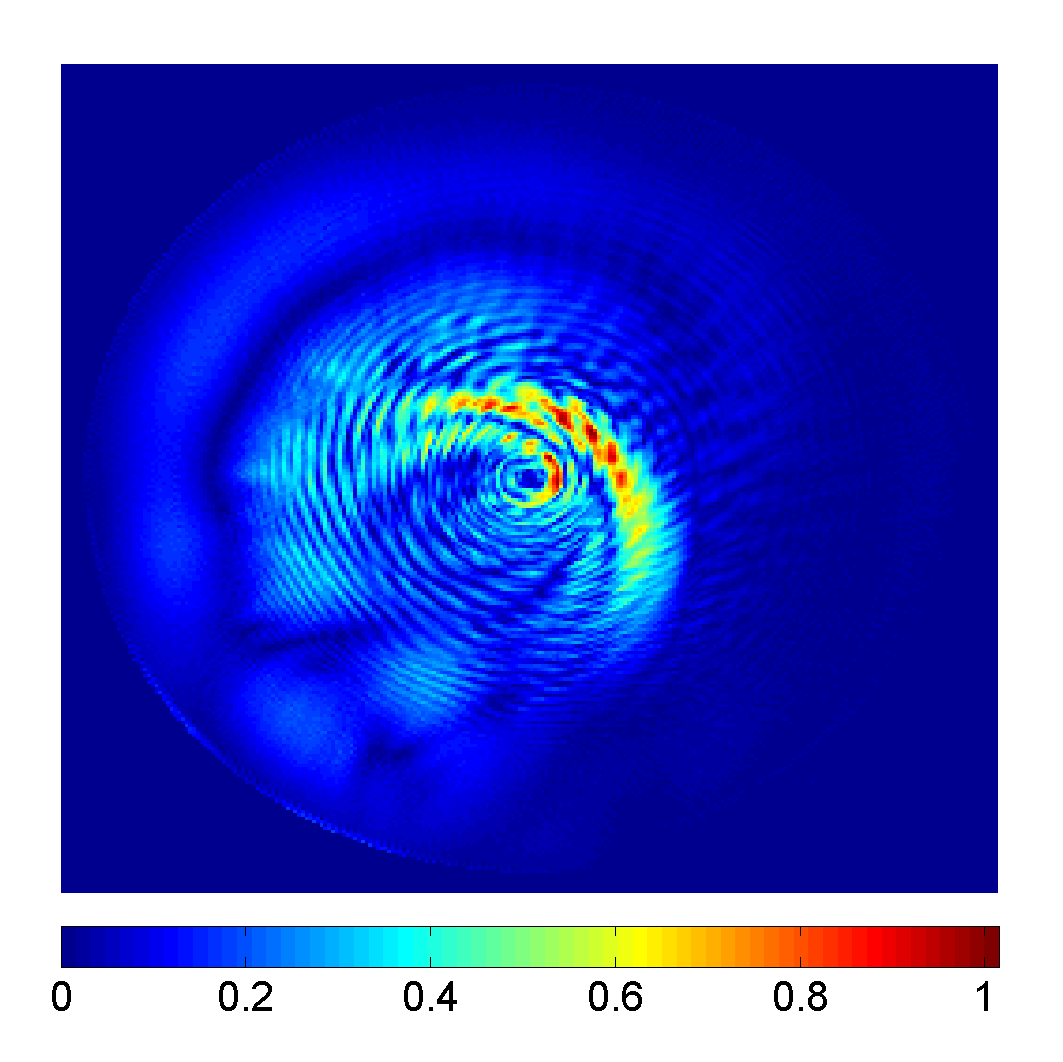}
    \label{fig:CNC_artifact}
    }
    \caption{\subref{fig:convplot1}\&\subref{fig:convplot2}: convergence plot for various values of the couple $(k,R)$. \subref{fig:CPC_artifact}\&\subref{fig:CNC_artifact}: pointwise errors $|f-f_{rc}|$ in cases where the (constant) curvature $\kappa$ and the tensor order $k$ are too high and the series fails to converge. In the positive curvature case (left), a smooth artifact is appearing. In the negative curvature case (right), the errors at gridscale suggest that geodesics do not sample the center of the domain well enough.}
    \label{fig:errorsat10}
\end{figure}

\clearpage
\subsubsection{A one-parameter family of non-constant curvature metrics} \label{ssec:oneparam}

In order to illustrate that the reconstruction algorithms also work for metrics with non-constant curvature, we now use the following one-parameter family of metrics first presented in \cite{Monard2013}
\begin{align}
    g_\ell(x,y) = \exp\left( \ell \exp\left( - ((x-.2)^2 + y^2) /2 \sigma^2  \right) \right), \quad x^2 + y^2 \le 1, \quad \sigma = 0.25.
    \label{eq:lensmetric}
\end{align}
As $\ell$ increases from $0$ to $e$, the metric transitions continuously from Euclidean to simple, to non-simple non-trapping, to trapping, see Fig. \ref{fig:lensmetric}.

\begin{figure}[htpb]
    \centering
    \includegraphics[trim = 26 40 27 40, clip, height=0.115\textheight]{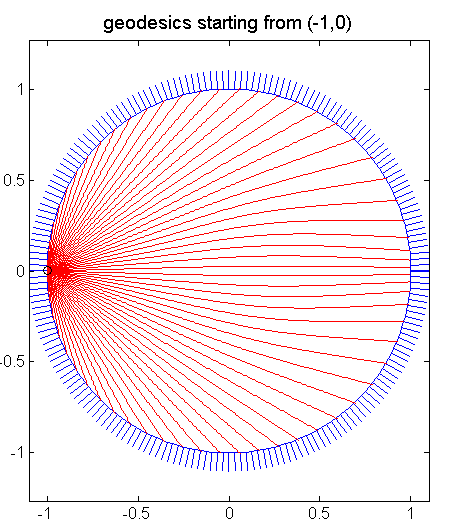} \qquad\qquad
    \includegraphics[trim = 26 40 27 40, clip, height=0.115\textheight]{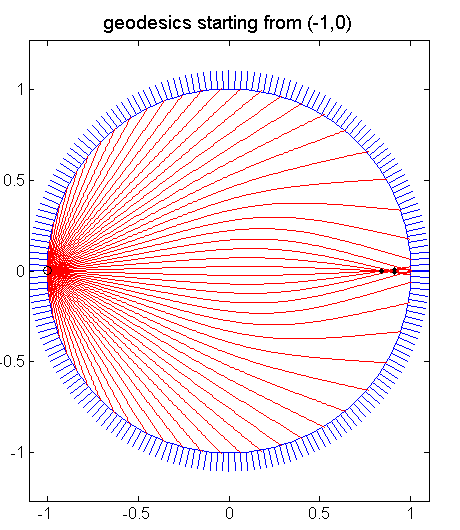} \qquad\qquad
    \includegraphics[trim = 26 40 27 40, clip, height=0.115\textheight]{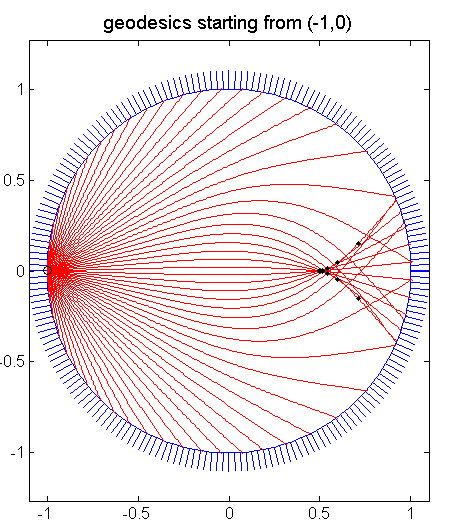}
    \caption{Examples of geodesics for the metric $g_\ell$ \eqref{eq:lensmetric} with $\ell \in \{0.3, 0.6, 1.2\}$ (l. to r.).}
    \label{fig:lensmetric}
\end{figure}

\begin{description}
    \item[Experiment 3: Inversion of $I_3$.] Using the differential order $k=3$ as an example, we run $10$ iterations of the series \eqref{eq:fhrc} reconstructing $f$ from the transform $I_3 (f) = I[f(x) e^{3i\theta}]$ with metric $g_\ell$ for $\ell$ taking the values $0.3$ (simple), $0.6$ (``close to'' simple yet non-simple) and $1.2$ (non-simple). For each value of $\ell$, the forward data $I_3 f$ appears on Fig. \ref{fig:lensIf}\subref{fig:lens_l03}\subref{fig:lens_l06}\subref{fig:lens_l12}.
    \item[Experiment 4: Inversion of $I_{3,\perp}$.] Using the same parameter values $(k,\ell) \in \{3\}\times (0.3,0.6,1.2)$ as in Experiment 3, we now implement the reconstruction algorithm based on formula \eqref{eq:hhrc}, i.e. we reconstruct the function $f(x)$ from knowledge of $I_{3,\perp} f = I [X_\perp (f(x) e^{3i\theta})]$, where $f(x) e^{3i\theta}\in \Omega_3$ and $X_\perp (f(x) e^{3i\theta}) \in \Omega_2 \oplus \Omega_4$. For each value of $\ell$, the forward data $I_{3,\perp} f$ appears on Fig. \ref{fig:lensIf}\subref{fig:duallens_l03}\subref{fig:duallens_l06}\subref{fig:duallens_l12}.
\end{description}

\noindent A convergence plot for both experiments 3\&4 is provided Fig. \ref{fig:lens_artifacts}\subref{fig:convplotlens} and Table \ref{tab:2} summarizes the relative $L^2$ errors after 10 iterations of each series. 

\begin{figure}[htpb]
    \centering
    \subfigure[$I_3 f$ (mod/arg), metric $g_{\ell}$ with $\ell = 0.3$]{
    \includegraphics[trim = 10 20 40 30, clip, width=0.48\textwidth]{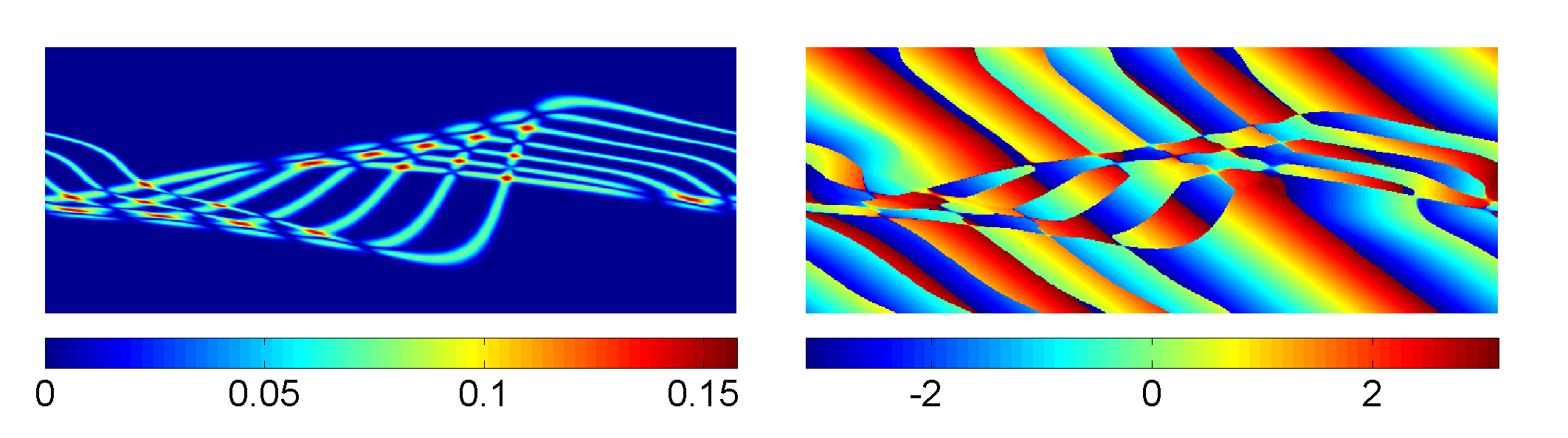}
    \label{fig:lens_l03}
    }
    \subfigure[$I_{3_\perp} f$ (mod/arg), metric $g_{\ell}$ with $\ell = 0.3$]{
    \includegraphics[trim = 10 20 40 30, clip, width=0.48\textwidth]{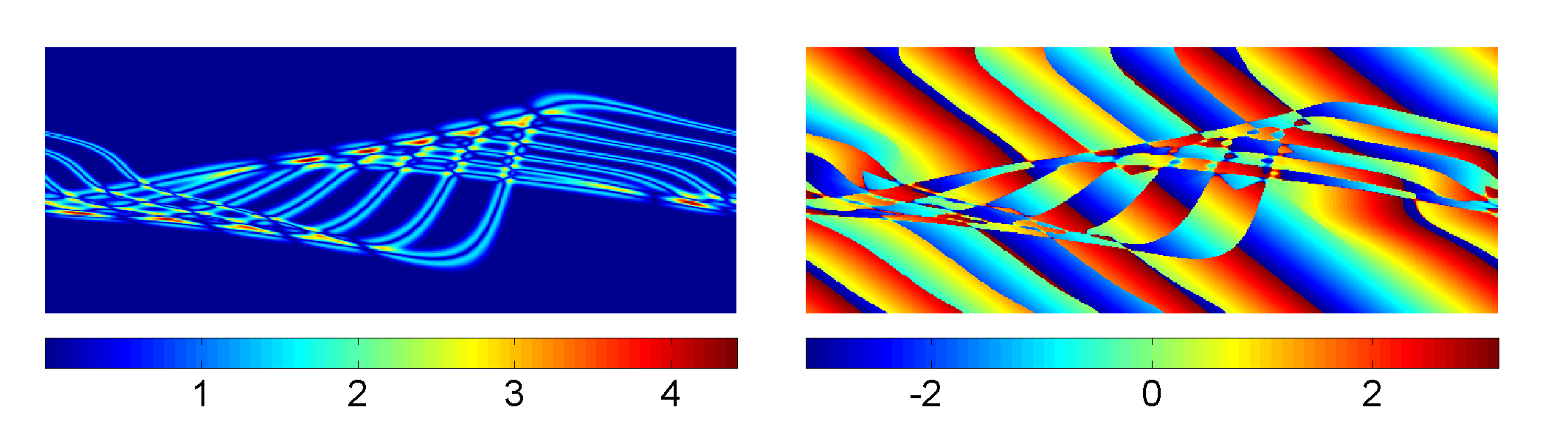}
    \label{fig:duallens_l03}
    }
    \subfigure[$I_3 f$ (mod/arg), metric $g_{\ell}$ with $\ell = 0.6$]{
    \includegraphics[trim = 10 20 40 30, clip, width=0.48\textwidth]{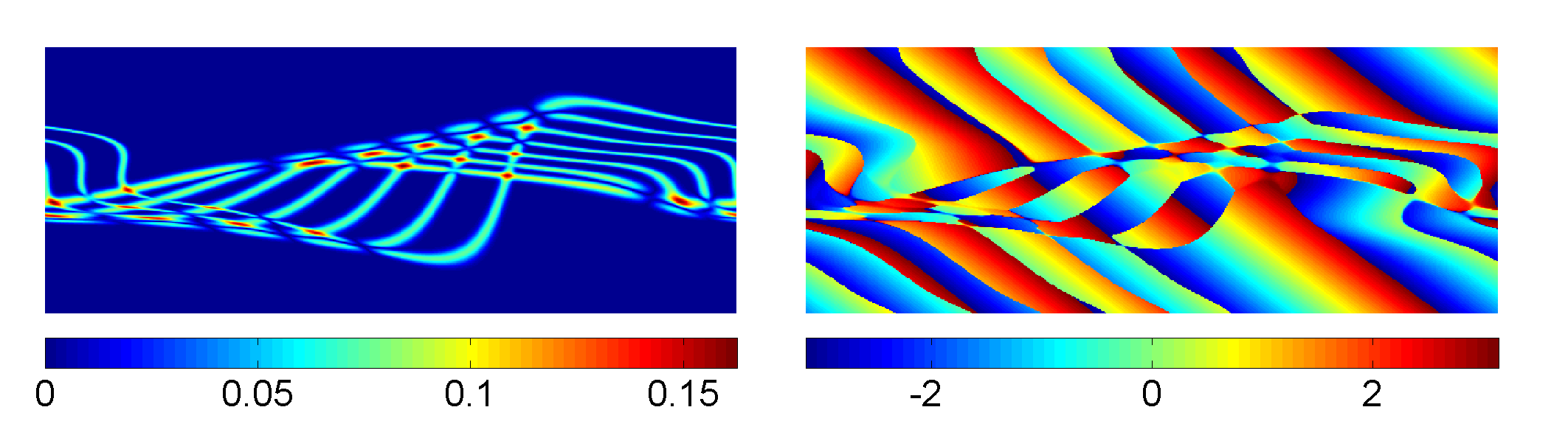}
    \label{fig:lens_l06}
    }
    \subfigure[$I_{3,\perp} f$ (mod/arg), metric $g_{\ell}$ with $\ell = 0.6$]{
    \includegraphics[trim = 10 20 40 30, clip, width=0.48\textwidth]{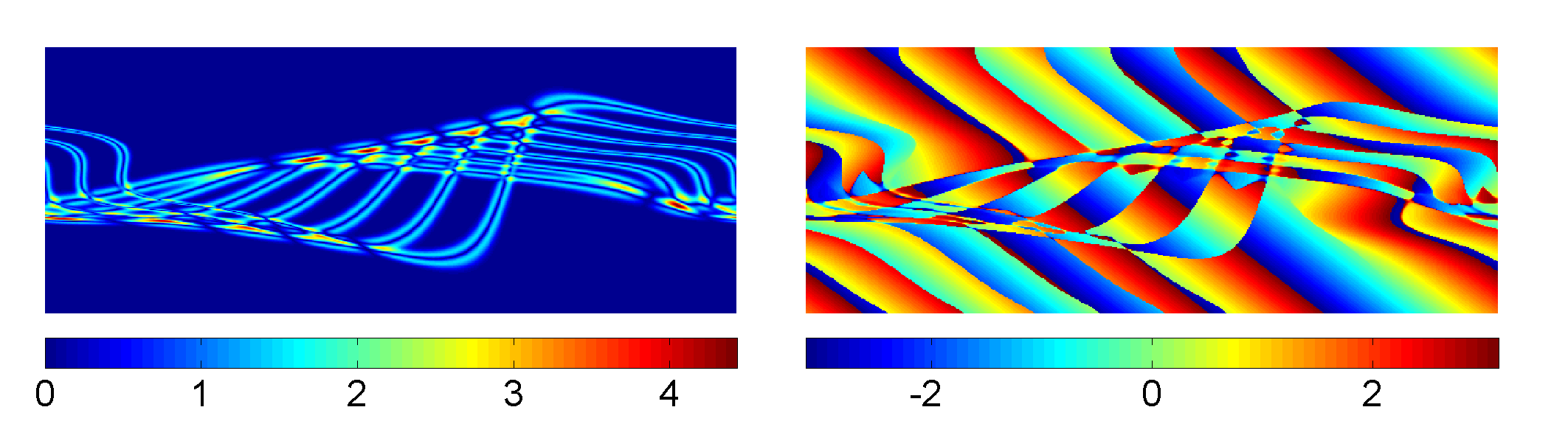}
    \label{fig:duallens_l06}
    }    
    \subfigure[$I_3 f$ (mod/arg), metric $g_{\ell}$ with $\ell = 1.2$]{
    \includegraphics[trim = 10 20 40 30, clip, width=0.48\textwidth]{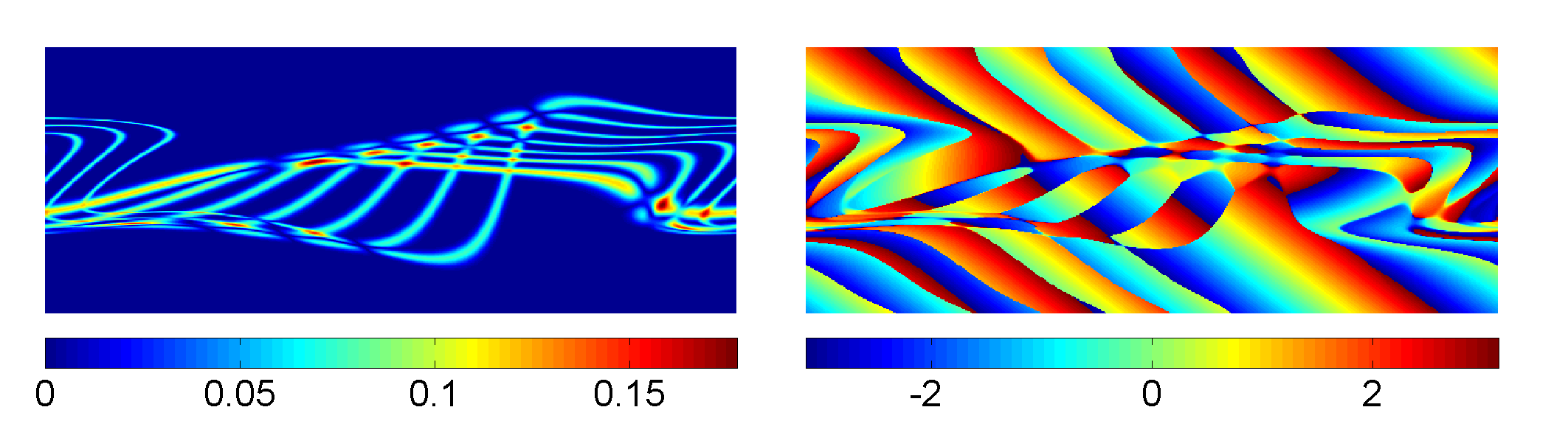}
    \label{fig:lens_l12}
    }
    \subfigure[$I_{3,\perp} f$ (mod/arg), metric $g_{\ell}$ with $\ell = 1.2$]{
    \includegraphics[trim = 10 20 40 30, clip, width=0.48\textwidth]{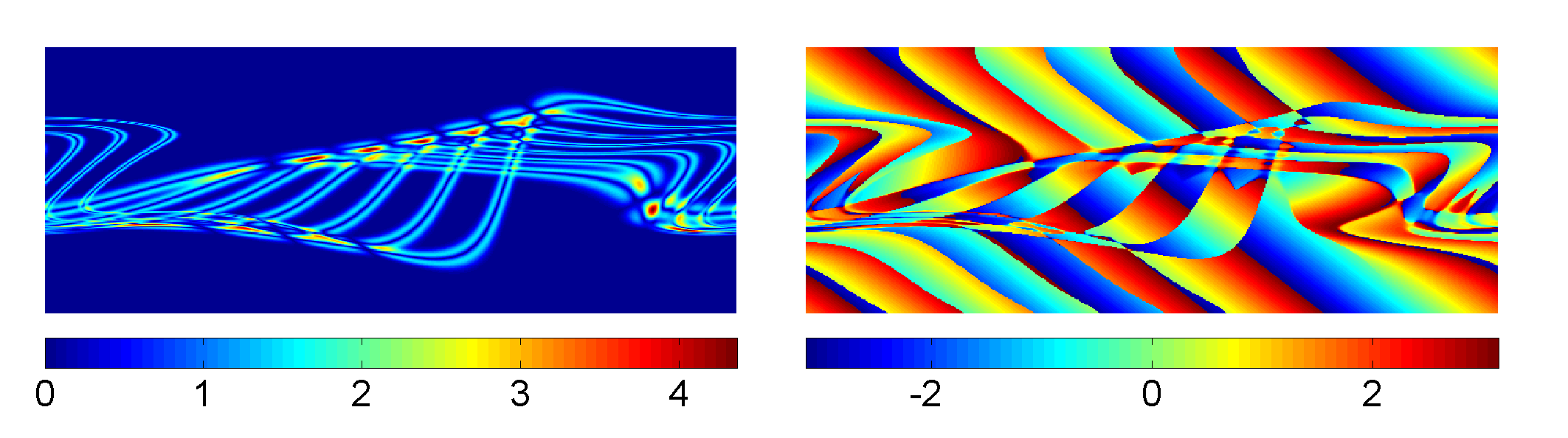}
    \label{fig:duallens_l12}
    }
    \caption{Examples of forward data $I_3 f$ and $I_{3,\perp}f$ for different metrics $g_\ell$. The data is complex-valued and each plot contains modulus and argument. Axes are $(\beta,\alpha)\in [0,2\pi]\times [-\frac{\pi}{2}, \frac{\pi}{2}]$.}
    \label{fig:lensIf}
\end{figure}

\begin{table}[htpb]
    \centering
    \begin{tabular}{|c||c|c|c|c|}
	\hline
	& $\ell = 0.3$ & $\ell = 0.6$ & $\ell = 0.9$ & $\ell = 1.2$ \\
	\hline
	\hline
	inversion of $I_3$ & 2.0 & 3.0 & 25 & 43.8 (DV) \\
	\hline
	inversion of $I_{3,\perp}$ & 1.3 & 1.8 & 22.7 (NC) & 38.7 (DV) \\
	\hline
    \end{tabular}
    \caption{$L^2$ relative error (in \%) at 10 iterations for Experiments 3 and 4. (NC) indicates that the series had not reached convergence yet was still converging. (DV) indicates divergence.}
    \label{tab:2}
\end{table}

\begin{figure}[htpb]
    \centering
    \subfigure[Inversion of $I_3$]{    
    \includegraphics[trim = 10 20 40 30, clip, height=0.30\textwidth]{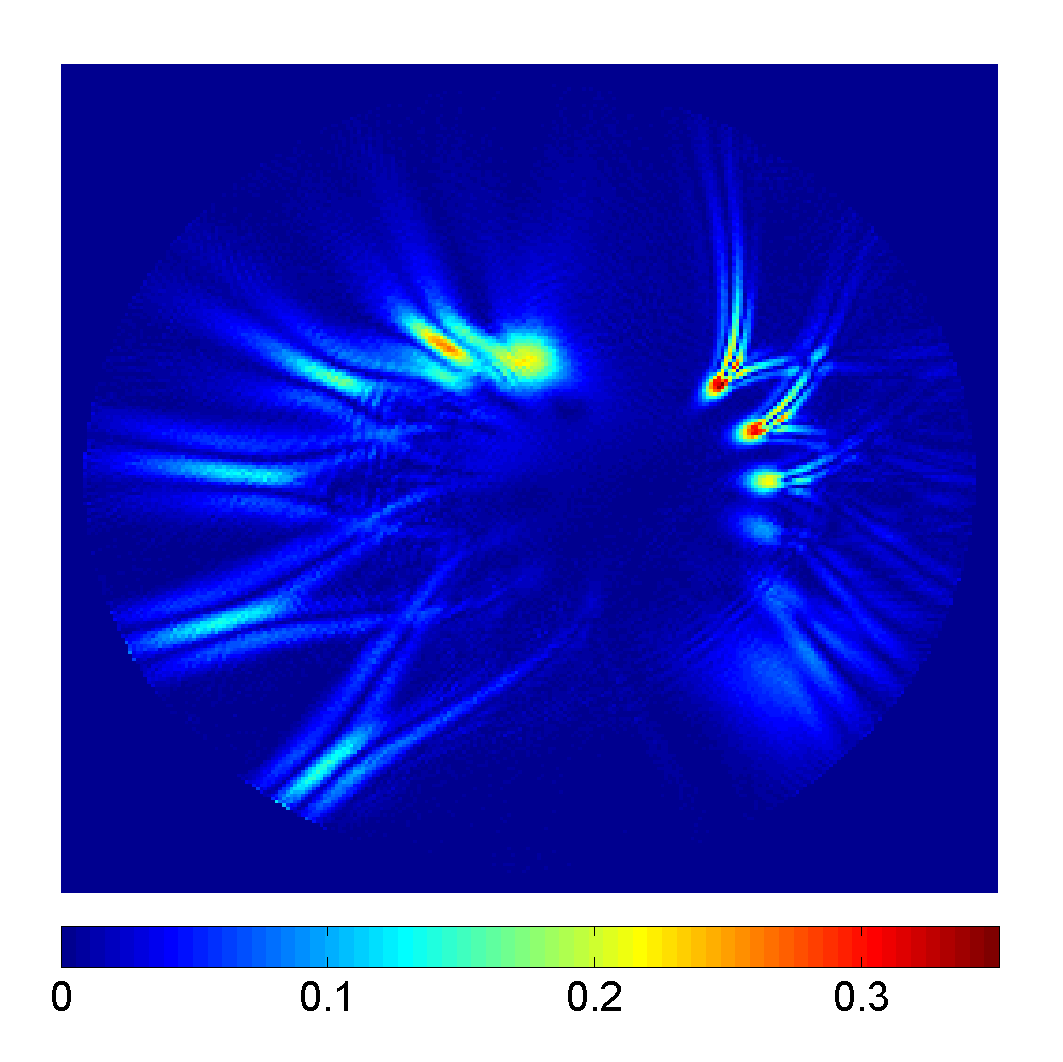}
    \label{fig:lens_artifact}
    }
    \subfigure[Inversion of $I_{3,\perp}$]{        
    \includegraphics[trim = 10 20 35 30, clip, height=0.30\textwidth]{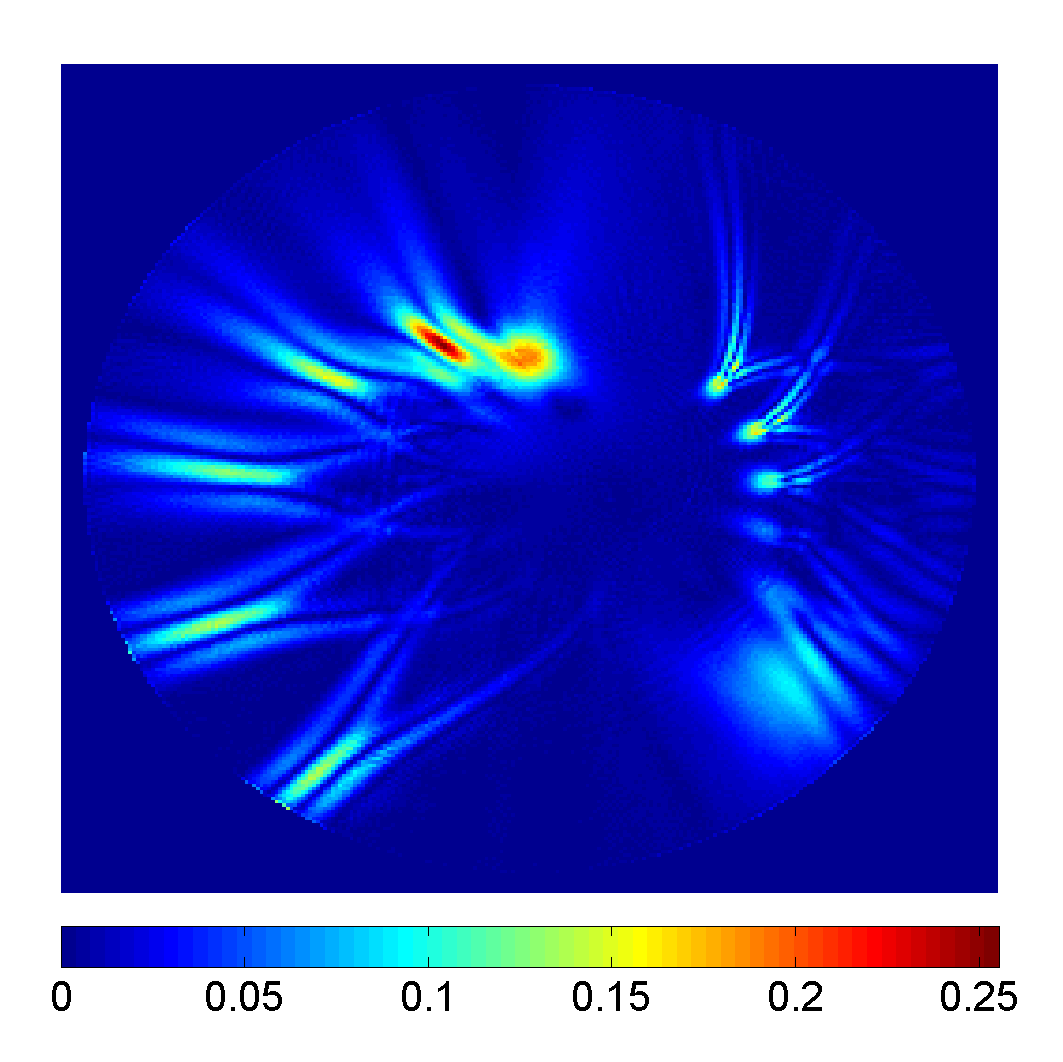}
    \label{fig:duallens_artifact}
    }
    \subfigure[Convergence plots]{        
    \includegraphics[trim = 5 0 15 25, clip, height=0.36\textwidth]{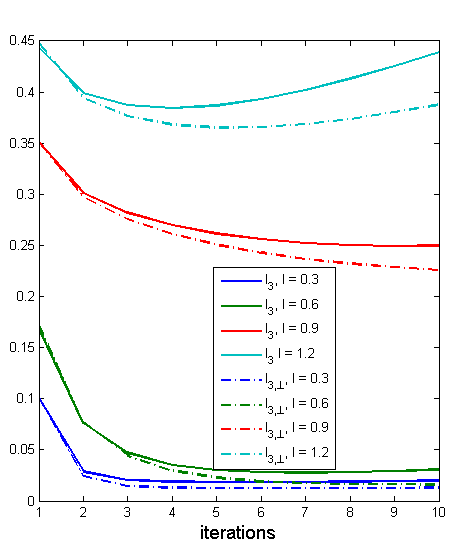}
    \label{fig:convplotlens}
    }    
    \caption{\subref{fig:lens_artifact}\&\subref{fig:duallens_artifact}: Pointwise error $|f-f_{rc}|$ at 10 iterations when the series fail to converge (case $\ell = 1.2$). \subref{fig:convplotlens}: $L^2$ convergence plots for experiments 3 and 4.}
    \label{fig:lens_artifacts}
\end{figure}

\clearpage

\subsection{Concluding comments on the numerical experiments}

In light of Experiments 1 and 2, a few remarks are in order: 
\begin{itemize}
    \item[$(i)$] As was observed previously in \cite{Monard2013}, negative curvature tends to scarcify geodesics at the center of the disk, which results in an undersampling of this area. This may be observed on Fig. \ref{fig:CCgeo}, where each plot contains the same number of geodesic curves, and for higher negative curvature, very few actually go near the center. This issue may be solved by choosing a discretization of the influx boundary $\partial_+ SM$ that is adapted to the ambient metric, though how to solve this issue in a way that is both theoretically sound and numerically practical will be investigated and presented in future work. Let us mention that this explains why the data in Fig. \ref{fig:lensIf}\subref{fig:CNC_I6f_R1p2} is narrowly supported near the axis $\alpha = 0$, and also partly explains the high-frequency undersampling artifacts observed on Fig. \ref{fig:errorsat10}\subref{fig:CNC_artifact}.

    \item[$(ii)$] Table \ref{tab:1} suggests that convergence of the series is successful when both the tensor order $k$ and curvature $\kappa$ are small enough, and that this convergence degrades as one increases either one of these parameters. This is in good agreement with estimate \eqref{eq:Wkest} in Proposition \ref{prop:Wkest}, where the estimate on the norm of the error operator degrades as either $k$ or $\kappa$ increases. In these cases, it is not expected that the operator $W_k$ remains a contraction, and the Neumann series needs not converge, leaving the following questions open:
	\begin{enumerate}
	    \item Is the operator $I + W_k^2$ still invertible as $k$ and/or $\kappa$ increases ?
	    \item If it is, how to invert it ? In particular, how to account for some of the eigenvalues of $W_k^2$ which now have modulus greater than one ?
	\end{enumerate} 
\end{itemize}

\noindent In light of Experiments 3 and 4, we make the following comments:
\begin{itemize}
    \item[$(iii)$] Aside from numerical inaccuracies, the convergence plots on Fig. \ref{fig:lens_artifacts}\subref{fig:convplotlens} demonstrate the successful convergence of both Neumann series to the initial function in cases $\ell = 0.3$ and $0.6$. It should be noted that the case $\ell = 0.6$ is a non-simple metric, thereby leaving hopes for theoretical improvements in certain cases of non-simple metrics.

    \item[$(iv)$] Similarly to point $(ii)$ above, the open questions posed there remain also valid as the metric becomes no longer simple yet non-trapping. Indeed, Fig.\ref{fig:lens_artifacts}\subref{fig:convplotlens} clearly demonstrate the failure of both Neumann series inverting $I_3$ and $I_{3,\perp}$ to converge when $\ell = 1.2$. Moreover, one observes in Fig. \ref{fig:lens_artifacts}\subref{fig:lens_artifact}-\subref{fig:duallens_artifact} the creation of artifacts at the conjugate loci of each gaussian bump that belongs to the initial phantom. Such artifacts, already observed in \cite{Monard2013} in the case of functions and solenoidal vector fields (i.e. the case $k=0$ here), will be the object of future work. 
\end{itemize}

\section*{Acknowledgements} 
The author would like to thank his postdoctoral mentor Gunther Uhlmann for encouragement and support, as well as Gabriel Paternain and Sean Holman for helpful discussions. This research is partially supported by NSF grant DMS-1025372. The author thanks the referees for suggesting reference \cite{Svetov2013}.

\bibliographystyle{siam}

\end{document}